\documentclass[12pt]{article}

\usepackage{latexsym,amsfonts,amsmath,epsfig,tabularx,amsthm,dsfont,mathrsfs}
\usepackage{enumerate} 
\usepackage{enumitem}
\usepackage{upref}

\newtheorem{Satz}{Satz}[section]
\newtheorem{Thm}[Satz]{Theorem}

\newtheorem{Corol}[Satz]{Corollary}
\theoremstyle{remark}

\theoremstyle{definition} 
\newtheorem{Def}[Satz]{Definition}
\newtheorem{Rem}[Satz]{Remark}
\newtheorem{Prop}[Satz]{Proposition}

\newcommand{\hra}{\ensuremath{\hookrightarrow}}
\newcommand{\vr}{\ensuremath{\varrho}}
\newcommand{\ol}{\ensuremath{\overline}}

\newcommand{\la}{\ensuremath{\lambda}}
\newcommand{\ve}{\ensuremath{\varepsilon}}
\newcommand{\vp}{\ensuremath{\varphi}}

\newcommand{\Lci}{\ensuremath{\overset{\circ}{L}}{}^r_{p}(\real^n)}

\newcommand{\Lciv}{\ensuremath{\overset{\circ}{L}}{}^r_{p}(\ell_q,\real^n)}
\newcommand{\Lcive}{\ensuremath{\overset{\circ}{L}}{}^r_{p}(\ell_q,\real)}

\newcommand{\Lv}{\ensuremath{L^r_{p}(\ell_q,\real^n)}}
\newcommand{\Lve}{\ensuremath{L^r_{p}(\ell_q,\real)}}

\newcommand{\Lr}{\ensuremath{L^r_{p}(\real^n)}}

\newcommand{\Lw}{\ensuremath{L_p (w_\alpha,\rn)}}

\newcommand{\Hv}{\ensuremath{H^{\vr} L_{p'} (\ell_{q'},\rn)}}
\newcommand{\Hve}{\ensuremath{H^{\vr} L_{p'} (\ell_{q'},\real)}}

\newcommand{\Hr}{\ensuremath{H^{\vr} L_{p'} (\rn)}}

\newcommand{\di}{\ensuremath{{\mathrm d}}}
\newcommand{\zn}{\ensuremath{{\mathbb Z}^n}}
\newcommand{\rn}{\ensuremath{{\mathbb R}^n}}
\newcommand{\real}{\mathbb{R}}

\newcommand{\nat}{\mathbb{N}}
\newcommand{\ganz}{\mathbb{Z}}

\newcommand{\eq}{equation}
\newcommand{\supp}{\ensuremath{\mathrm{supp \,}}}

\title{\textbf{On the boundedness of singular integrals in Morrey spaces and its preduals}} 

\author{
Marcel Rosenthal \& Hans-J{\"u}rgen Schmeisser\\
Universit\"at Jena, Mathematisches Institut\\
Ernst-Abbe-Platz 2, 07743 Jena, Germany\\
email: marcel.rosenthal@uni-jena.de $\cdot$ mhj@uni-jena.de}

\begin{document}

%

\maketitle

\begin{abstract}
We reduce the boundedness of operators in Morrey spaces $L_p^r(\rn)$, its preduals, $H^{\vr}L_p (\rn)$, and their preduals $\Lci$ to the boundedness of the appropriate operators in Lebesgue spaces, $L_p(\rn)$. Hereby, we need a weak condition with respect to the operators which is satisfied for a large set of classical operators of harmonic analysis including singular integral operators and the Hardy-Littlewood maximal function. 
The given vector-valued consideration of these issues is a key ingredient for various applications in harmonic analysis.
\end{abstract}
\\
\textbf{Math Subject Classifications: Primary 42B35, 46E30, 42B15, 42B20; Secondary 42B25.}

\section{Introduction}
Let $\Lci$ be the completion of $D(\rn)$ in $L_p^r(\rn)$, where
	\begin{align*} 
		\left\|f|{ L}^r_{p}(\real^n)\right\|= & \sup_{x\in\real^n} \sup_{R>0} R^{-\left(\frac{n}{p}+r\right)} \left\|f|L_p(B_R(x))\right\|, 1<p<\infty, -\frac{n}{p}\le r <0.
	\end{align*}
Then we have
\begin{\eq} \label{GiF1}
  \left(\Lci\right)''\cong \left(H^{\vr}L_{p'} (\rn)\right)'\cong L_p^r(\rn),
\end{\eq}
where the second duality assertion is due to \cite{Zor86, Kal98, AX04, GM13, RoT14_2} and the first assertion is observed by \cite{AX12} and proved by \cite{RoT14_2}.
Roughly speaking in this paper we prove that the $L_p(\rn)$-boundedness of an operator $T$ satisfying the condition 
\begin{equation} \label{p30:GG}
  |(Tf)(y)|\leq c \int_{\real^n} \frac{|f(z)|}{|y-z|^n} \di z \quad \text{ for all } f\in D(\rn) \text{ and } y\notin \supp(f), 
\end{equation}
implies its boundedness in $\Lci$. Therefrom, under some additional conditions with respect to $T$ we get  also the boundedness of $T$ in $H^{\vr}L_p (\rn)$ and $L_p^r(\rn)$ by \eqref{GiF1} and duality arguments. 
Our paper can be considered as an extension of the new approach given in \cite{RT13} and \cite{RoT14_2} to a wider class of operators and to the vector-valued situation. Let us mention that the extension of operators
of this type and related norm estimates have to be treated with greater care than in many related papers investigating mapping properties of operators in $L_p^r(\rn)$. We refer to Remark \ref{LS} for the relation of our paper 
to the existing literature. In particular, we cannot expect an unique extension to Morrey spaces $L_p^r(\rn)$. On the contrary it turned out that there are infinitely many possible extension operators (cf. \cite[Remark 5.3]{RoT14_2}).
Let us also mention that the vector-valued situation under consideration is crucial
having in mind applications as a Michlin-H\"ormander type theorem (and hence applications to Navier-Stokes equations cf. \cite{Tr12} and \cite[Remark 4.3]{RT13}), Littlewood-Paley theory for Morrey spaces and its preduals as well as for Lizorkin representations of Triebel-Lizorkin-Morrey spaces.
The given results are partially contained in \cite{Ros13}.
Condition \eqref{p30:GG} is due to Soria and Weiss \cite{SW94} who transferred the boundedness of singular operators on Lebesgue spaces to the boundedness of these operators in some weighted Lebesgue spaces. 
\par
The paper is organized as follows. Basic definitions and preliminaries which are needed later on are collected in Section 2. Duality theory for vector-valued Morrey-type spaces is treated in Section 3. The main results can be found in Theorem \ref{TDp1:GG} (preduals of Morrey spaces) and Theorem \ref{TDp2:GG} (Morrey spaces as bidual spaces). In final Section 4 we prove our main results concerning the transference of mapping properties of operators satisfying condition \eqref{p30:GG} to vector-valued Morrey type spaces. The general theorem is presented in Subsection 4.1 (Theorem \ref{vB:thm}) following the method developed in \cite{RT13} and \cite{RoT14_2}.
As a consequence of our main theorem we obtain mapping properties for various classes of operators in vector-valued Morrey-type spaces. Subsection 4.2 is concerned with Calder\'{o}n-Zygmund operators. Here we present also an alternative approach via weighted spaces (Theorem \ref{vB2:thm}). Maximal operators of Hardy-Littlewood and Calder\'{o}n-Zygmund type as well as related vector-valued inequalities are considered in Subsection 4.3. The final Subsection 4.4 is devoted to some classes of Fourier multipliers such as characteristic functions, smooth multipliers and Bochner-Riesz mulipliers at the critical index. 

\section{Definitions and Preliminaries}
\subsection{Notation}

We use standard notation. Let $\nat$ be the collection of all natural numbers and $\nat_0 = \nat \cup \{0 \}$. Let $\rn$ be the Euclidean $n$-space, where
$n\in \nat$. Put $\real = \real^1$. 
Let $S(\rn)$ be the Schwartz space of all complex-valued rapidly decreasing infinitely differentiable functions on $\rn$ and let $S' (\rn)$ be the space of all tempered distributions on $\rn$.
Let $D(\rn) = C^\infty_0 (\rn)$ be the collection of all infinitely differentiable complex-valued  functions 
with compact support in $\rn$, where the support of a function $f$ is abbreviated by $\supp(f)$. 
Moreover, denotes $C(\rn)$ and $\text{Lip}(\rn)$ the collection of all continuous and Lipschitz continuous, respectively, and bounded complex-valued functions defined on $\rn$.
Furthermore, $L_p (\rn)$ with $1\le p <\infty$, is the standard complex Banach space with respect to the Lebesgue measure,
normed by
\[
\| f \, | L_p (\rn) \| = \Big( \int_{\rn} |f(x)|^p \, \di x \Big)^{1/p}.
\]
For a measurable subset $M$ of $\rn$ we similarly define $L_p (M)$.
Moreover, $|M|$ stands for the Lebesgue measure of $M$ and $\chi_M$ for the characteristic function on $M$. As usual $\ganz$ is the collection of all integers; and $\zn$ where $n\in \nat$ denotes the
lattice of all points $m= (m_1, \ldots, m_n) \in \rn$ with $m_j \in \ganz$. 
As usual, $L_p^{\mathrm{loc}}(\rn)$ collects all equivalence classes of almost everywhere coinciding measurable complex locally $p$-integrable functions, hence $f\in L_p (M)$ for any bounded measurable set $M$ in $\rn$.
For any $p\in (1,\infty)$ we denote by $p'$ the conjugate index, namely, $1/p+1/{p'}= 1$.
For Banach spaces $X$ and $Y$ and an operator $T:X\rightarrow Y$ \[T:X\hra Y\]
means, that the operator is bounded, that is,
\[
  \left\|Tx|Y\right\|\le c \left\|x|X\right\| 
\]
where the the constant $c$ is independent of $x\in X$.
Let $D(\rn)\hra X$.
A bounded operator $\widetilde{T}$ acting in $X$, hence $\widetilde{T}:X\hra X$, is called an extension of $T$ to $X$ if it coincides on $D(\rn)$ with $T$. 
We denote the Fourier transform of $f$ on $S(\rn)$ or $S'(\rn)$ by $\hat{f}$ and its inverse by $\check{f}$ where the normalisation of $\hat{f}$ does not matter for our estimates. 
The concrete value of constants may vary from one
formula to the next, but remains the same within one chain of (in)equalities. Finally, $A\cong B$ is an abbreviation that there are two constants $c$, $C>0$ such that $c A \leq B \leq C A$. 

\subsection{Morrey spaces, duals and preduals}

\begin{Def}{} \label{d1:def}
\upshape 
	For $1< p< \infty$ and $-\frac{n}{p}\leq r<0$ we define \textit{Morrey spaces} as  
	\begin{equation*}
				{ L}^r_{p}(\real^n)\equiv\{f \in L_p^{\text{loc}}(\real^n) \mbox{ : } \left\|f|{ L}^r_{p}(\real^n)\right\|<\infty\}
	\end{equation*}
	with the 	norm  
	\begin{align*} 
		\left\|f|{ L}^r_{p}(\real^n)\right\|\equiv & 
		\sup_{M\in\ganz^n} \sup_{J\in\ganz} 2^{J\left(\frac{n}{p}+r\right)} \left\|f|L_p(Q_{JM})\right\|
		\\ \cong &
		\sup_{x\in\real^n} \sup_{R>0} R^{-\left(\frac{n}{p}+r\right)} \left\|f|L_p(B_R(x))\right\|
				, 
	\end{align*}	
		where $Q_{JM}\equiv Q_{J,M}\equiv 2^{-J}\left(M+\left[-1,1\right]^n\right)$ and $B_R(x)$ denotes the ball with radius $R$ centered at $x$. \\
		Moreover, $\stackrel{\circ}{{ L}^r_{p}}(\rn)$ denotes the closure of $D(\rn)$ with respect to $\left\|\cdot|{L}^r_{p}(\real^n)\right\|$.
\end{Def}
\begin{Def}  \label{D2.3}
Let $1<p<\infty$ and $-n <\vr <-n/p$. 
Then the \textit{predual Morrey spaces} $H^{\vr} L_p (\rn)$ collects all $h\in S'(\rn)$ 
which can be represented as
\begin{align}   \label{2.9}
\begin{split}
&h= \sum_{J\in \ganz, M\in \zn} \lambda_{J,M} a_{J,M} \quad \text{in} \quad S'(\rn) \quad \text{with} 
\\ &\supp a_{J,M} \subset {Q_{J,M}},\qquad \| a_{J,M} \, | L_p (\rn) \|\le 2^{-J(\frac{n}{p} + \vr)}, 
\end{split}
\end{align}
such that
\begin{\eq}    \label{2.10a}
\sum_{J\in \ganz, M \in \zn} |\lambda_{J,M}| <\infty.
\end{\eq}
Furthermore,
\[ 
\| h \, | H^{\vr} L_p (\rn) \| \equiv \inf \sum_{J\in \ganz, M \in \zn} |\lambda_{J,M}|
\] 
where the infimum is taken over all representations \eqref{2.9}, \eqref{2.10a}.
\end{Def}
\begin{Rem} \label{GF3}
The notation of $H^{\vr} L_p (\rn)$ as a predual will be justified in the Theorem \ref{TDp1:GG}. 
  By triangular and H\"older's inequality \eqref{2.9} and \eqref{2.10a} ensure that the convergence in \eqref{2.9} is unconditionally in $L_u(\rn)$, where $\vr u =-n$. In particular it holds $H^{\vr} L_p (\rn) \hra L_u(\rn)$  and we have $1<u<p$ (cf. \cite[(3.10)]{RoT14_2}). 
	Let $L_p (\rn, w_\alpha)$ with $1< p <\infty$ and $w_\gamma (x)
= (1+|x|^2)^{\gamma/2}$, $\gamma \in \real$, be the weighted Lebesgue spaces, normed by
\begin{\eq} \label{GiE}
\| f \, | \Lw\| = \| w_\alpha f \, | L_p (\rn) \|. 
\end{\eq} 
Then it holds
\begin{\eq} \label{74:GGE}
  L_{p} (w_\alpha,\rn) \hra H^{\vr} L_p (\rn)
\end{\eq} 
with $\alpha>n/{p'}$  (cf. \cite[(3.5)]{RoT14_2}).
Furthermore, $D(\rn)$, $S(\rn)$ are dense both in $\Lci$ and $H^{\vr} L_p (\rn)$. $H^{\vr} L_p (\rn)$ and $\Lr$ are Banach spaces and $L_u(\rn)\hra\Lr\hra\Lw$ for $u=-n/r$ and $\alpha<-n/p$ (cf. \cite[Thm. 3.1]{RoT14_2}). The last embedding as well as \eqref{74:GGE} can be sharpened cf. \eqref{74:GG} below. 
\end{Rem}
\begin{Def} \label{GTTF} Let $1 < p <\infty$, $-n <\vr <-n/p$.
 	Let $H^{\vr} L_p (\rn)_F^\ve$ be the following subspace of $H^{\vr} L_p (\rn)$ defined as
\begin{align*}
  H^{\vr} L_p (\rn)_F^\ve \equiv & \left\{ \vp\in H^{\vr} L_p (\rn) \left| \text{ there exists an } L \in \nat \right.\right. \\ &\left. \left. 
\quad \text{ such that } \vp= \sum_{\substack{J\in \ganz, M\in \zn\\ |J|\le L, |M|\le L}} {h_{J,M}},\ \supp h_{J,M}\subset Q_{J,M} \text{ and }   
	\right.\right. \\ & \left. \quad
	\sum_{\substack{|J|\le L\\ |M|\le L}} 2^{J(\frac{n}{p} + \vr)} \| h_{J,M} \, | L_p (Q_{J,M})\|\le (1+\ve) \| \vp | H^{\vr} L_p (\rn) \| 
	\right. \}.
\end{align*}
\end{Def} 
\begin{Prop} \label{PttL}
Let $1 < p <\infty$, $-n <\vr <-n/p$.
Then 
$H^{\vr} L_p (\rn)_F^\ve$ is dense in $ H^\vr L_p (\rn) $. 
\end{Prop}
\begin{proof}    
Let $h\in H^{\vr} L_p (\rn) $ and $\ve>0$.
Let $h= \sum_{J\in \ganz, M\in \zn} \lambda_{J,M} a_{J,M}$ in $S'(\rn)$ such that $\sum_{J\in \ganz, M \in \zn} |\lambda_{J,M}| \le (1+\ve/2) \| h \, | H^{\vr} L_p (\rn) \|$ with $\supp a_{J,M} \subset {Q_{J,M}}$ and $ \| a_{J,M} \, | L_p (\rn) \|\le 2^{-J(\frac{n}{p} + \vr)}$. 
We define then $h_{J,M}\equiv\lambda_{J,M} a_{J,M}$ for $J\in \ganz$, $M\in \zn$ and obtain
\begin{align*}
  \sum_{J\in \ganz, M \in \zn} 2^{J(\frac{n}{p} + \vr)} \| h_{J,M} \, | L_p ( Q_{J,M})\|
	&\le \sum_{J\in \ganz, M \in \zn} |\lambda_{J,M}| \\&\le \left(1+\frac{\ve}{2}\right) \| h \, | H^{\vr} L_p (\rn) \|.
\end{align*}
Let
\[ 
h^L = \sum_{|J|\le L, |M| \le L} h_{J,M}, \qquad L \in \nat.
\] 
Then
\[ 
\| h - h^L \, | H^{\vr} L_p (\rn) \| \to 0 \qquad \text{if} \quad L \to \infty.
\]
Hence,
\begin{align*}
  \sum_{\substack{|J|\le L\\ |M|\le L}} 2^{J(\frac{n}{p} + \vr)} \| h_{J,M} \, | L_p ( Q_{J,M})\| &\le \left(1+\frac{\ve}{2}\right) \| h \, | H^{\vr} L_p (\rn) \| \\&\le (1+\ve) \| h^L \, | H^{\vr} L_p (\rn) \|.
\end{align*}
\end{proof} 

\subsection{Vector-valued Morrey spaces}

\begin{Def} \label{24:GG}
Let $1< p< \infty$, $-\frac{n}{p}\leq r<0$ and $1< q<\infty$. Let $L^r_{p}(\ell_q, \real^n)$ be the collection 
of all sequences of functions $f_j$ belonging to $L^r_{p}(\rn)$ such that
  \[
    \left\|f_j|{ L}^r_{p}(\ell_q, \real^n)\right\| \equiv 
		\left\|\left\{f_j\right\}|{ L}^r_{p}(\ell_q, \real^n)\right\| \equiv
		\left\| \left. \left(\sum_{j=0}^\infty \left|f_j(\cdot) \right|^q \right)^{\frac{1}{q}}\right|{ L}^r_{p}(\real^n) \right\|<\infty.  
  \]
Moreover,
   	\begin{align*}
		\stackrel{\circ}{{L}^r_{p}}(\ell_q, \real^n) \equiv &\left\{  \left. \left\{f_j\right\}_{j\in\nat_0} \in {L}^r_{p}(\ell_q, \real^n) \right| \text{ there exist } f_j^k\in D(\rn) \text{ for all } \right. \\ & \quad \left. j\in\nat_0, k\in\nat \text{ and } f_j^k=0 \text{ for } j>k \text{ with }
		\right. \\ & \quad \left.
		\left\|\left.\left\{f_j-f_j^k\right\}_{j}\right|{L}^r_{p}(\ell_q,\rn)\right\|\rightarrow 0 \quad (k\rightarrow\infty) \right\}. 
	\end{align*}
Furthermore, for $\alpha\in\real$ we define the space $L_{p}(\ell_{q}, w_{\alpha},\rn)$ as
$L^r_{p}(\ell_{q},\rn)$ using the norm of $\Lw$ instead the norm of $\Lr$. If $\alpha=0$, we simply write $L_{p}(\ell_{q},\rn)$.
\end{Def}
\begin{Def}
Then $H^{\vr} L_p (\ell_{q},\rn)$ denotes
the collection of all sequences of functions $g_j$ belonging to $H^{\vr} L_p (\rn)$ 
such that
$\left\|g_j(\cdot)|\ell_{q}\right\|$ is in $H^{\vr} L_p (\rn)$.
Moreover, $H^{\vr} L_p (\ell_{q},\rn)_F^\ve$ stands for  
the collection of all sequences of functions $g_j$ belonging to $H^{\vr} L_p (\ell_{q},\rn)$ 
such that
$\left\|g_j(\cdot)|\ell_{q}\right\|$ is in $H^{\vr} L_p (\rn)_F^\ve$.
\end{Def}

\section{Duals and preduals - the vector-valued case} 
\subsection{Predual spaces}
The duality with respect to Morrey spaces is discussed in the scalar case in detail with complete proofs in \cite{RoT14_2}. Here we give complete proofs in the vector-valued case following their approach.

\begin{Thm} \label{TDp1:GG}
Let $1<p<\infty$, $-\frac{n}{p}< r<0$, $r+\vr=-n$ and $1<q<\infty$. 
Then the predual space of $L_p^r(\ell_q,\rn)$ is $H^{\vr} L_{p'} (\ell_{q'},\rn)$. 
Moreover,
\[g\in \left(H^{\vr} L_{p'} (\ell_{q'},\rn)\right)'\] if, and only if, it can be uniquely represented as 
\begin{\eq} \label{87:GG}
  g(f)=\int_{\rn} \sum_{j\in\nat_0} g_j(x) f_j(x)\di x
\end{\eq}
for all $f\equiv\{f_j \}\in L_{p'} (\ell_{q'}, w_{\alpha},\rn)  \hra H^{\vr} L_{p'} (\ell_{q'},\rn)$, $\alpha>n/p$, where
\[
  \{g_j\}\in L_p^r(\ell_q,\rn) \text{ and } \left\|g\left|\left( \Hv \right)'\right.\right\|=\left\|g_j|L_p^r(\ell_q,\rn)\right\|.
\]
Moreover, if $\{g_j\}\in L_p^r(\ell_q,\rn)$, then
\begin{\eq} \label{1:GHM}
  \left\|g_j|L_p^r(\ell_q,\rn)\right\|=\sup_{f}\left|\int_{\rn}\sum_{j\in\nat_0} g_j(x) f_j(x)\di x\right| 
\end{\eq}
where the supremum is taken over all $f \equiv \{f_j\}\in H^{\vr} L_{p'} (\ell_{q'},\rn)$ with \\ $\left\|f|H^{\vr} L_{p'} (\ell_{q'},\rn)\right\|\leq 1$. 
\end{Thm}
\begin{proof}
Let $g\equiv\{g_j\}\in L_p^r(\ell_q,\rn)$ and $\{\tilde{f_j}\}\in H^{\vr} L_{p'} (\ell_{q'},\rn)$ such that
$\left\|\left.\tilde{f_j}(\cdot)\right|\ell_{q'}\right\|$ is in $H^{\vr} L_{p'} (\rn)_F^\ve$. H\"older's inequality yields 
\begin{align*}
  &\int_{\rn} \sum_{j\in\nat_0} \left|g_j(y) \tilde{f_j}(y)\right|\di y
  \leq \int_{\rn} \left\|\{g_j(y)\}_j|\ell_q\right\| \left\|\{\tilde{f_j}(y)\}_j|\ell_{q'}\right\| \di y
	\\ \leq & \sum_{\substack{J\in \ganz, M\in \zn\\ |J|\le L, |M|\le L}} \int_{\rn} \left\|\{g_j(y)\}_j|\ell_q\right\| h_{J,M}(y) \di y
  \\ \leq & 
	\sum_{\substack{J\in \ganz, M\in \zn\\ |J|\le L, |M|\le L}} 2^{J(\frac{n}{p} + r)} \left\| \left.\left\|\{g_j(\cdot)\}_j|\ell_q\right\| \right| L_{p}(Q_{J,M}) \right\|
   2^{J(\frac{n}{p'} + \vr)} \left\| h_{J,M} | L_{p'}(Q_{J,M})\right\|
	\\ \leq & (1+\ve) \left\|g_j|L_p^r(\ell_q,\rn) \right\| \left\|\tilde{f_j}|H^{\vr} L_{p'} (\ell_{q'},\rn) \right\|
\end{align*}
where $\left\|\left.\tilde{f_j}(\cdot)\right|\ell_{q'}\right\|$ is represented as in Definition \ref{GTTF} and $r+\vr+n=0$.
Therefore the operator $T_{g}$ given by
\[
  T_{g}(\{\tilde{f_j}\})\equiv \int_{\rn} \sum_{j\in\nat_0} \left| g_j(y) \tilde{f_j}(y) \right| \di y
\]
is bounded on $H^{\vr} L_{p'} (\ell_{q'},\rn)_F^\ve$.
We get the (unique) continuous extension $T_{g}:H^{\vr} L_{p'} (\ell_{q'},\rn)\hra \real$ by means of Proposition \ref{PttL}, where this extension is justified as in the linear case cf. \eqref{ToY} and \eqref{FGM} below.
Let $\{f_j\}\in \Hv$. 
By Proposition \ref{PttL} there is furthermore a sequence $\{f^k_j\}$ of $H^{\vr} L_{p'} (\ell_{q'},\rn)_F^\ve$ such that $\{f^k_j\}$ tends to $\{f_j\}$ in $\Hv$ for $k\rightarrow\infty$. 
For $u$ such that $\vr u =-n$ by $\Hr\hra L_u(\rn)$ exists a subsequence such that $\left\|\left.f^{k_l}_j(\cdot)-f_j(\cdot)\right|\ell_{q'}\right\|\rightarrow 0$ almost everywhere with respect to the Lebesgue measure in $\rn$ for $l\rightarrow\infty$. This implies $f^{k_l}_j \rightarrow f_j$ almost everywhere for all $j\in\nat_0$ if $l\rightarrow\infty$.
The Lemma of Fatou yields then
\begin{align*}
&\int_{\rn} \sum_{j\in\nat_0} \left| g_j(y) f_j(y) \right| \di y
=  \int_{\rn} \sum_{j\in\nat_0} \left| g_j(y) \lim_{l\rightarrow\infty} f^{k_l}_j(y) \right| \di y \\ \le & 
\lim_{l\rightarrow\infty} T_{g}(\{{f^{k_l}_j}\}) =T_{g}(\{f_j\})
\end{align*}
Thus, for $\ve\searrow 0$ 
we obtain 
\begin{align} \begin{split} \label{58:GG}
  \left|\int_{\rn} \sum_{j\in\nat_0} g_j(y) f_j(y)\di y\right| 
	&\le\int_{\rn} \sum_{j\in\nat_0} \left|g_j(y) f_j(y)\right| \di y 
	\\&\le \left\|g_j|L_p^r(\ell_q,\rn) \right\| \left\|f_j|H^{\vr} L_{p'} (\ell_{q'},\rn) \right\|
\end{split} \end{align}
for $\{g_j\}\in L_p^r(\ell_q,\rn)$ and $\{f_j\}\in H^{\vr} L_{p'} (\ell_{q'},\rn)$. 
Hence, in particular, any $\{g_j\}\in L_p^r(\ell_q,\rn)$ induces a bounded linear functional on $H^{\vr} L_{p'} (\ell_{q'},\rn)$.
\par Conversely, suppose that $g$ is a bounded linear functional on \\$H^{\vr} L_{p'} (\ell_{q'},\rn)$ with the norm $\left\|g\right\|$. 
Taking into account \eqref{74:GGE} the linear functional $g$ induces a bounded linear functional on $L_{p'} (\ell_{q'}, w_{\alpha},\rn)$ for $\alpha>n/p$ and therefore we have the representation formula
\begin{\eq} \label{TTLCTN}
  g(\{f_j\})=\int_{\rn} \sum_j g_j(y) f_j(y) \di y 
\end{\eq}
for some $\{g_j\}\in L_{p}(\ell_{q},w_{-\alpha},\rn)$ and for all $\{f_j\}\in L_{p'}(\ell_{q'}, w_{\alpha},\rn)$. 
Let $\{\tilde{f_j}\}\in L_{p'}(\ell_{q'}, w_{\alpha},\rn)$ with $\supp \tilde{f_j} \subset Q_{J,M}$ for all $j\in\nat_0$. Then
\[
  \left\|\tilde{f_j}|H^{\vr} L_{p'} (\ell_{q'},\rn) \right\| \le 2^{J(\frac{n}{p'} +\vr)} 
	\left\| \left.\|\{\tilde{f_j}(\cdot)\}_j|\ell_{q'}\| \right| L_{p'}(Q_{J,M}) \right\|.
\]
With $\frac{n}{p'} +\vr = -\frac{n}{p} -r$ one obtains
\begin{align*} 
  |g(\{f_j\})| &\le \left\|g\right\| \left\|\left.\tilde{f_j} \right| H^{\vr} L_{p'} (\ell_{q'},\rn) \right\|
	\\&\le \left\|g\right\| 2^{-J(\frac{n}{p} +r)} 
	\left\| \left.\|\{\tilde{f_j}(\cdot)\}_j|\ell_{q'}\| \right| L_{p'}(Q_{J,M}) \right\|.
\end{align*} 
Then one has by duality in $L_{p'} (\ell_{q'},Q_{J,M})$ and \eqref{TTLCTN}
\[
\| g_j \, | L_{p} (\ell_q,Q_{J,M}) \| \le 2^{-J (\frac{n}{p'} +r)} \| g \|.
\]
Hereby, $L_{p'} (\ell_{q'},Q_{J,M})$ is defined similarly as
$L^r_{p}(\ell_{q'},\rn)$ using $L_{p'}(Q_{J,M})$ instead of $\Lr$. Note also that an element of $L_{p'} (\ell_{q'},Q_{J,M})$, say $\{\tilde{f_j}\}$, is also in $L_{p'}(\ell_{q'}, w_{\alpha},\rn)$ if one extends $\tilde{f_j}$, $j\in\nat_0$, outside of $Q_{J,M}$ by zero. The last inequality proves $\{g_j\}\in L_p^r(\ell_q,\rn) $ and
\[
  \left\|g_j|L_p^r(\ell_q,\rn) \right\| \le \| g \|.
\] 
%
%
\end{proof} 

\subsection{Dual spaces}
In the proof of the next theorem, which is a vector-valued extension of \cite[Thm. 4.1, (4.5)]{RoT14_2}, we benefit from the following general assertion.

\begin{Prop}[page 73 of \cite{ET96} and Lemma in Section 1.11.1 of \cite{T78}] \label{86:GG}
  Let $\{ A_j \}_{j\in\nat_0}$ be a sequence of complex Banach spaces and $\{A_j'\}_{j\in\nat_0}$ their respective duals. Moreover, we put
   	\begin{align*}
		c_0(\{A_j\}) &\equiv \left\{  \left. a\equiv\left\{a_j\right\}_{j\in\nat_0} \right| a_j\in A_j,  \right.\\ &\left. 
		\quad\left\|a|c_0(A_j)\right\|\equiv\left\|a|\ell_\infty(A_j)\right\|\equiv\sup_j\left\|a_j|A_j\right\|<\infty, \left\|a_j|A_j\right\|\rightarrow 0\right\},
		\\
		\ell_1(\{A_j'\}) &\equiv \left\{  \left. a'\equiv\left\{a_j'\right\}_{j\in\nat_0} \right| a_j'\in A_j', \left\|a'|\ell_1(A_j')\right\|\equiv\sum_j\left\|a_j|A_j'\right\|<\infty \right\}. 
	\end{align*}
Then
\begin{align*}
  &\left(c_0(\{A_j\})\right)'=\ell_1(\{A_j'\}) \text{ with } a'(a)= \sum_{j=0}^\infty a_j'(a_j) \text{ and }\\ &\left\|\cdot\left| (c_0(A_j))'\right. \right\| =\left\|\cdot \left| \ell_1(A_j') \right.\right\|.
\end{align*}
\end{Prop}

\begin{Thm} \label{TDp2:GG}
Let $1<p<\infty$, $-\frac{n}{p}< r<0$, $r+\vr=-n$ and $1<q<\infty$.
Then the dual space of $\Lciv$ is $\Hv$. 
Moreover,
$g\in \left(\Lciv\right)'$ if, and only if, it can be uniquely represented as 
\[
  g(f)=\int_{\rn} \sum_{j\in\nat_0} g_j(x) f_j(x)\di x
\]
for all $f\equiv\{f_j \} \in L_{-\frac{n}{r}}(\ell_q,\rn) \hra \Lciv$, where
\[
  \{g_j\}\in \Hv \text{ and } \left\|g\left|\left(\Lciv\right)'\right.\right\|=\left\|g_j|\Hv\right\|.
\]
Moreover, if $\{g_j\}\in \Hv$, then
\begin{\eq} \label{2:GHM}
  \left\|g_j|\Hv\right\|=\sup_{f}\left|\int_{\rn}\sum_{j\in\nat_0 } g_j(x) f_j(x)\di x\right| 
\end{\eq}
where the supremum is taken over all $f \equiv \{f_j\} \in \Lciv$ with \\$\left\|f\left| \Lv\right.\right\|\leq 1$. 
\end{Thm}

\begin{proof}
It follows from \eqref{58:GG}  that any $\{g_j\}\in \Hv$ induces a bounded linear functional on $\Lciv$. 
\par Conversely, suppose $g$ is a bounded linear functional on $\Lciv$ with norm $\left\|g\right\|$. 
We observe that
\begin{align*}
    \left\|\{f_j\}|\Lv\right\|&= \sup_{J\in\ganz,M\in\ganz^n} \Big( \int_{Q_{J,M}} \left(\sum_{j\in\nat_0}|f_j(x)|^q\right)^\frac{p}{q} 2^{J(n +pr)} \, \di x \Big)^\frac{1}{p} \\&= \left\|f^j_{JM}|c_0\left(L_p(\ell_q,\mu_J,Q_{JM})\right)\right\|,
\end{align*}
  where $f^j_{JM}\equiv f_j \chi_{Q_{JM}}$, $\mu_J(\di x)\equiv 2^{J(n +pr)}$ and 
  \begin{align*}
    &\left\|f^j_{JM}|c_0\left(L_p(\ell_q,\mu_J,Q_{JM})\right)\right\|
		\\ &\equiv\sup_{J\in\ganz,M\in\ganz^n} \Big( \int_{Q_{JM}} \left(\sum_{j\in\nat_0}|f^j_{JM}(x)|^q\right)^\frac{p}{q} 2^{J(n +pr)} \, \di x \Big)^\frac{1}{p}. 
  \end{align*}
  	This shows 
	that $\Lciv$ is isomorphic to a closed subspace of \\$c_0\left(L_p(\ell_q,\mu_J,Q_{JM})\right)$ analogously to the scalar-valued case in \cite[(4.18)-(4.20)]{RoT14_2}.
	More precisely, we have a linear, surjective and isometric map $I:\{f_j\}\mapsto \{f^j_{JM}\}$ from $\Lciv$ onto the closed subspace $\{\{f^j_{JM}\} | \{f_j\}\in \Lciv \}$ of $c_0\left(L_p(\ell_q,\mu_J,Q_{JM})\right)$ and 
	\[
	  I \Lciv = \{\{f^j_{JM} \}| \{f_j\}\in \Lciv \} \hra c_0\left(L_p(\ell_q,\mu_J,Q_{JM})\right).
	\]
  Hahn-Banach's theorem yields
$g\in \left(\Lciv\right)'$ if, and only if, \\$g\in \left(c_0\left(L_p(\ell_q,\mu_J,Q_{JM})\right)\right)'$ and by Proposition \ref{86:GG} 
we have the representation
\begin{align} \label{TYL}
  g(\{f_j\})&=\sum_{J\in\ganz,M\in\ganz^n} \int_{Q_{JM}} \sum_{j\in\nat_0}  f_j(x) g^j_{JM}(x) 2^{J(n +pr)} \di x
\end{align}
 for any $\{f_j\}\in \Lciv$ with $\{ g^j_{JM} \}\in \ell_1\left(L_{p'}(\ell_{q'},\mu_J,Q_{JM})\right)$, where
 \begin{align*}
   &\left\|\{ g^j_{JM} \}|\ell_1\left(L_{p'}(\ell_{q'},\mu_J,Q_{JM})\right)\right\|
   \\=& 
   \sum_{J\in\ganz,M\in\ganz^n} \Big( \int_{Q_{JM}} \left(\sum_{j\in\nat_0}|g^j_{JM}(x)|^{q'}\right)^\frac{p'}{q'} 2^{J(n +pr)} \, \di x \Big)^\frac{1}{p'}.
 \end{align*}
Moreover, Hahn-Banach's theorem implies that 
\begin{align*} 
   &\left\|g\left|\left(\Lci\right)'\right.\right\|=\inf  \left\{ \left\|g^j_{JM}\left|\ell_1\left(L_{p'}(\ell_{q'},\mu_J,Q_{JM})\right)\right.\right\| \left| g(\{f_j\})= g^j_{JM} (\{f_j\})  \right.\right.\\& \qquad \left.\left. \text{ for all }\{f_j\}\in \Lciv  \text{ and } g^j_{JM}\in\ell_1\left(L_{p'}(\ell_{q'},\mu_J,Q_{JM})\right)  
   \right.\right\}.
\end{align*} 
Using Lebesgue's dominated convergence theorem we deduce from \eqref{TYL} (cf. \eqref{TYL2} for an integrable majorant) the representation 
\begin{align*} 
  g(\{f_j\})&= 
	\int_{\rn} \sum_{j\in\nat_0} f_j(x) \sum_{J\in\ganz,M\in\ganz^n} g^j_{JM}(x) \chi_{Q_{JM}}(x) 2^{J(n +pr)} \di x
\end{align*}
for $\{f_j\}\in L_{-\frac{n}{r}}(\ell_q,\rn)$.
Let $\ve>0$. For $h^j_{JM}\equiv g^j_{JM} \chi_{Q_{JM}} 2^{J(n +pr)}$ we obtain
\[
  \left\|g^j_{JM}|L_{p'}(\ell_{q'},\mu_J,Q_{JM}) \right\|=2^{-J\left(\frac{n}{p}+r\right)} \left\|h^j_{JM}|L_{p'}(\ell_{q'},\rn)\right\| \equiv \la_{JM}.
\]
Therefore $\{\la_{JM}\}_{J,M} \in \ell_1$ and for an appropriate choice of $g^j_{JM}$ we obtain also
$\left\|\la|\ell_1\right\| \le (1+\ve) \left\|g\right\|$.
For $a^j_{JM}$ given by $h^j_{JM}=\la_{JM} a^j_{JM}$ we have then 
$\left\|\left.a^j_{JM} \right|L_{p'}(\ell_{q'},\rn)\right\|\le 2^{J\left(\frac{n}{p}+r\right)}$ with $\supp\left(a^j_{JM}\right) \subset {Q_{JM}}$. Finally, it holds 
$\left\{\sum_{J\in\ganz,M\in\ganz^n} \la_{JM} a^j_{JM}\right\}_j\in \Hv$ and 
\[
\left\{\sum_{J\in\ganz,M\in\ganz^n} g^j_{JM} \chi_{Q_{JM}} 2^{J(n +pr)}\right\}_j\in \Hv.
\]
Indeed, we have 
\[
  \left(\sum_{j\in\nat_0} \left|\sum_{\substack{J\in \ganz, \\ M\in \zn}} \la_{JM} a^j_{JM}\right|^{q'} \right)^\frac{1}{q'} \le \sum_{\substack{J\in \ganz, \\ M\in \zn}} \la_{JM} \left(\sum_{j\in\nat_0} \left| a^j_{JM}\right|^{q'} \right)^\frac{1}{q'} \in \Hr
\]
using $b_{JM} \equiv \left(\sum_{j\in\nat_0} \left| a^j_{JM}\right|^{q'} \right)^\frac{1}{q'}$ with $\supp\left(b_{JM}\right) \subset {Q_{JM}}$ and \\$\left\|b_{JM}|L_{p'}(\rn)\right\|\le 2^{J\left(\frac{n}{p}+r\right)}=2^{-J\left(\frac{n}{p'}+\vr\right)}$. Finally, 
\begin{align*}
  \left\|\left. \left\{\sum_{J\in\ganz,M\in\ganz^n} g^j_{JM}(x) \chi_{Q_{JM}}(x) 2^{J(n +pr)} \right\}_j \right|\Hv\right\| &\le \left\|\la|\ell_1\right\| \\&\le (1+\ve) \left\|g\right\|.
\end{align*}
By the same argumentation we obtain also
\[
\left\{\sum_{J\in\ganz,M\in\ganz^n} \left|g^j_{JM} \chi_{Q_{JM}}\right| 2^{J(n +pr)}\right\}_j\in \Hv\hra 
L_{-\frac{n}{\vr}}(\ell_q,\rn).
\]
Together with $\{f_j\}\in L_{-\frac{n}{r}}(\ell_q,\rn)$ and H\"older's inequality 
\begin{\eq} \label{TYL2}
  \sum_{j\in\nat_0} |f_j| \sum_{J\in\ganz,M\in\ganz^n} \left|g^j_{JM} \chi_{Q_{JM}}\right| 2^{J(n +pr)}
\end{\eq}
is an integrable majorant.
Moreover, we observe $L_{-\frac{n}{r}}(\ell_q,\rn)\hra \Lciv$. Indeed, for $\{f_j\}\in L_{-\frac{n}{r}}(\ell_q,\rn)$ there is a sequence $\{f^k_j\}_j$ tending to $\{f_j\}$ in $L_{-\frac{n}{r}}(\ell_q,\rn)$ as 
$k\rightarrow\infty$ with $f_j^k\in D(\rn)$ and $f_j^k=0$ for $j>k$ (and $f_j^k \nearrow f_j$ as 
$k\rightarrow\infty$) which also implies $\{f^k_j\}_j\rightarrow\{f_j\}$ in $\Lciv$ as 
$k\rightarrow\infty$ by $L_{-\frac{n}{r}}(\rn)\hra\Lr$.
\end{proof}

\section{Mapping properties of operators}
\subsection{The main theorem}

Next we extend the approach developed in \cite{RT13} and \cite{RoT14_2} to a wider class of operators and to vector-valued spaces.


\begin{Prop} \label{84A:GG} 
  Let $1<p<\infty$, $-\frac{n}{p}< r<0$ and $1<q<\infty$. Then $\Lciv$ coincides with the completion of finite sequences of continuous compactly supported functions. More precisely, it holds
 	\begin{align*}
		\Lciv = &\left\{  \left. \left\{f_j\right\}_{j\in\nat_0} \in {L}^r_{p}(\ell_q, \real^n) \right| \text{ there exist } f_j^k\in C(\rn) \text{ compactly}  \right. \\ & \quad \left. \text{supported for all } j\in\nat_0, k\in\nat \text{ and } f_j^k=0 \text{ for } j>k \text{ with } 
		\right. \\ & \quad \left.
		\left\|\left.\left\{f_j-f_j^k\right\}_{j}\right|{L}^r_{p}(\ell_q,\rn)\right\|\rightarrow 0 \right\}
		\equiv \overline{C_0(\ell_q, \real^n)}^{\|\cdot| \Lr \|}. 
	\end{align*}
\end{Prop}

\begin{proof}
  Let $\{f_j\}\in \overline{C_0(\ell_q, \real^n)}^{\|\cdot| \Lr \|}$. Let $\ve>0$. Then there exists a sequence $\{g_j\}$ with $g_j\in C(\rn)$ compactly supported with $g_j=0$ for $|j|>k$ and some $k\in\nat$ such that $\left\|f_j -g_j|\Lv\right\|<\ve$.
	Let $x\in\real^n$, $R>0$.   
   Let $\bar{R}>1$ such that $\supp \sum_{j=0}^k |g_j|^q \subset B_{\bar{R}-1}(0)$. Let $y\in\rn$ with $|y|<1$.
   Moreover, for $R\ge \bar{R}$
 	\begin{align*}
    \left\|g_j(\cdot-y)-g_j(\cdot)|L_p(\ell_q,B_R(x))\right\|\le \frac{\ve \bar{R}^r}{c} \left|B_{\bar{R}}(0)\right|^\frac{1}{p}\leq \ve R^{\frac{n}{p}+r}
  \end{align*}
	whenever
	\begin{\eq} \label{uc}
	  \sum_{j=0}^k |g_j(z-y)-g_j(z)| < \frac{\ve \bar{R}^r}{c} \qquad \text{ for all } z\in\rn
	\end{\eq}
	which holds by the uniform continuity of $g_j$, $j=0,\ldots,k$, for $|y|<\delta=\delta(\ve,\bar{R},r,g_0,\ldots,g_k)$
	where $c$ is a constant depending on $n$.
 	Furthermore, for $R< \bar{R}$ again by \eqref{uc}
 	\begin{align*}
    \left\|g_j(\cdot-y)-g_j(\cdot)|L_p(\ell_q,B_R(x))\right\|\le \frac{\ve \bar{R}^r}{c} \left|B_{R}(x)\right|^\frac{1}{p}\leq \ve R^{\frac{n}{p}+r}
  \end{align*} 	
\par
Let 
$\psi\in D(\rn)$ with $\supp \psi \subset B_1(0)$, $\int_{\rn} \psi (y)\di y=1$, $0\le \psi \le 1$ and $\psi_l(\cdot)\equiv l^n\psi(l\cdot)$, $l\in\nat$. Then it holds $\left\|\{g_j*\psi_l -g_j\}_j|\Lv\right\|<\ve$ for $l$ sufficient large 
where $g_j*\psi_l \in D(\rn)$, $j\in\nat_0$. Indeed, 
by means of Minkowski's inequality and the properties of $\psi_l$ we find
 	\begin{align*}
    &\left\|g_j*\psi_l-g_j|L_p(\ell_q,B_R(x))\right\|
    \\ \leq &  R^{\frac{n}{p}+r} \int_{|y|\le \frac{1}{l} }|\psi_l(y)|  R^{-\left(\frac{n}{p}+r\right)} \left\|g_j(\cdot-y)-g_j(\cdot)|L_p(\ell_q,B_R(x))\right\| \di y
    \le \ve R^{\frac{n}{p}+r} 
 	\end{align*}  
where $l$ is sufficiently large (depending on $\ve$).
\end{proof}

\begin{Rem}
  In the last Proposition we adapted the proof in  scalar-valued case $\Lr$ given in \cite[Proposition 3]{Zor86} to the vector-valued situation.
\end{Rem}

\begin{Thm} \label{vB:thm}
	Let $1<p<\infty$, $-\frac{n}{p}< r<0$, $r+\vr=-n$, $1<q<\infty$
	and let $\left\{T_j\right\}_{j\in\nat_0}$ be a sequence of operators  with the following properties:
\begin{enumerate}[label=(\roman{*}), ref=(\roman{*})]
	\item $T_j:D(\rn)\rightarrow C(\rn)$, $j\in\nat_0$, and $T_j$, $j\in\nat_0$, are 
	\begin{enumerate}
	\item either linear or 
	\item 	 
		\begin{align} \begin{split} \label{ToY}
	  &(T_j (f_1 +f_2))(y) \le (T_j f_1)(y)+(T_j f_2)(y),\\ &(T_j f)(y) = (T_j (-f))(y) ,\ T_j 0=0  \end{split}
	  	\end{align}
	  	for $f$, $f_1$, $f_2 \in D(\rn)$ and $ y\in\rn$;
	\end{enumerate}
	\item  
	 we have 
\begin{equation} \label{43a:GG}
  |(T_j f)(y)|\leq c_1 \int_{\real^n} \frac{|f(z)|}{|y-z|^n} dz 
\end{equation}
for all $f\in D(\rn)$ and all $y\notin \supp(f)$, where $c_1$ does not depend on $j\in\nat_0$, $f$ and $y$;
	\item \label{vb2c:GG} there is a constant $c_2$ such that 
		 \[
		 \left\|T_j f_j|L_{p}(\ell_q, \rn)\right\|
			\leq c_2 \left\|f_j|L_{p}(\ell_q, \rn)\right\|
	\] 
	for all $\{f_j\}_{j\in\nat_0}\subset D(\rn)$.
\end{enumerate}
Then, the following statements hold true.
\begin{enumerate} 
\item \label{60:GG}
There are unique continuous and bounded extensions $\widetilde{T_j}$ of $T_j$ to $\Lci$ for $j\in\nat_0$ such that 
\[
\left\{\widetilde{T_j}\right\}_{j\in\nat_0}: \Lciv\hra\Lciv .
\] 
\item \label{60b:GG}
If $T_j$ are linear for $j\in\nat_0$, then the dual operators of the unique linear and bounded extensions $\widetilde{T_j}$ of $T_j$ to $\Lci$, $\widetilde{T_j}': \Hr\hra\Hr$, 
satisfy \[\left\{\widetilde{T_j}'\right\}_{j\in\nat_0}: \Hv\hra\Hv.\]
If the extensions of $T_j$ to $L_p(\rn)$ due to assumption \ref{vb2c:GG} are formally self-adjoint for all $j\in\nat_0$, then $\widetilde{T_j}'$ are the unique linear and bounded extensions of $T_j$ acting in $\Hr$.
\item \label{59:GG}
If $T_j$ are linear for $j\in\nat_0$, then there are linear and bounded extensions $\widetilde{T_j}$ of $T_j$ to $\Lr$ such that 
\[ 
\left\{\widetilde{T_j}\right\}_{j\in\nat_0}: \Lv\hra \Lv .
\] 
\end{enumerate}
\end{Thm}

\begin{proof}
\textit{Step 1}.
We start showing Assertion \ref{60:GG}.

At first we will show that $\left\{T_j\right\}_{j\in\nat_0}: \Lciv\hra\Lv$. 
			Let $\{f_j\}_{j=0}^\infty\in \Lciv$ with $f_j\in D(\rn)$ for all $j$. 
	 Let $x\in\real^n$ and $R>0$.
	We decompose 
	\begin{equation*} 
		{f_j}=f_j^0+\sum_{i=1}^\infty{f_j^i},
	\end{equation*}
		where $f_j^0\equiv\varphi_0 f_j$ and $f_j^i\equiv\varphi_i f_j$ for $i,j\in \nat$ with $\{\varphi_i\}_{i\in\nat_0}\subset D(\rn)$ such that
	\[
	  \varphi_0= 1 \text{ on } B_{2R}(x), \qquad   \supp\varphi_0\subset  B_{4R}(x)
	\]
	and 
	\[
	  \supp\varphi_i\subset B_{2^{i+2}R}(x)\setminus{B_{2^{i}R}(x)}, \qquad \sum_{i\in\nat_0}\varphi_i=1. 
	\]	
 	By means of \ref{vb2c:GG} we 	obtain 
	\begin{equation*}
		\begin{split}
			& \left(\int_{B_R(x)}{\left(\sum_{j=0}^\infty\left|T_j {f}_j^0(y)\right|^q\right)^\frac{p}{q} \di y}\right)^\frac{1}{p} 
		\leq c {R}^{n\left(\frac{1}{p}+\frac{r}{n}\right)}
		\left\|f_j|{L}^r_{p}(\ell_q, \rn)\right\| .
		\end{split}
	\end{equation*}	
Let $i \in\nat$ and $y\in B_R(x)$. It follows from \eqref{43a:GG} that 
	\begin{equation*} 
		\begin{split}
			\left(\sum_{j=0}^\infty\left|T_j {f}_j^i(y)\right|^q\right)^\frac{1}{q} 
			\leq & c {(2^{i-1} R)}^{-n} \int_{\real^n}\left(\sum_{j=0}^\infty{\left|{f}_j^i(z)\right|^q}\right)^\frac{1}{q} \di z.
		\end{split}
	\end{equation*}	
	H\"older's inequality yields
	\begin{align*} 
& \left(\int_{B_{R}(x)}{\left(\sum_{j=0}^{\infty}\left|T_j \left( \sum_{i=1}^{\infty}f_j^i\right)(y)\right|^q\right)^\frac{p}{q} \di y}\right)^\frac{1}{p}
\\\leq & c \sum_{i=1}^\infty {(2^{i-1} R)}^{-n} \int_{\real^n}\left(\sum_{j=0}^\infty{\left|{f}_j^i(z)\right|^q}\right)^\frac{1}{q} \di z \ |B_{R}(x)|^\frac{1}{p}
\\\leq & c' \sum_{i=1}^\infty {(2^{i-1} R)}^{-n} {R}^\frac{n}{p} ({2^{i+2} R})^{n(1-\frac{1}{p})} 
({2^{i+2} R})^{\left(\frac{n}{p}+r\right)}
		\left\|f_j|{L}^r_{p}(\ell_q, \rn)\right\|
\\\leq & c'' {R}^{\left(\frac{n}{p}+r\right)} \left\|f_j|{L}^r_{p}(\ell_q, \rn)\right\|.
	\end{align*}	
	By subadditivity of the operators we obtain 
	\[
	  \left\|T_j f_j|\Lv\right\|\le c\left\|f_j|\Lv\right\|
	\]
	where $c$ does not depend on $\{f_j\}_j$.
	We get the unique continuous extension $T:\Lciv\hra\Lv$ of $\{T_j \}_j$
	whenever $T_j$ are linear. 
	If $T_j$ fulfills \eqref{ToY}, then we have for $f_1$, $f_2\in D(\rn)$, $y\in\rn$, $j\in\nat_0$
	\[
	  |(T_j f_1)(y)-(T_j f_2)(y)|\le (T_j (f_1-f_2))(y) 
	\]
	and hence for $\{f_j\}$, $\{\widetilde{f_j}\} \in\Lciv$ with $f_j$, $\widetilde{f_j}\in D(\rn)$ for all $j\in\nat_0$
	\begin{align} \begin{split}\label{FGM}
	  \left\|T_j f_j-T_j \widetilde{f_j}|\Lv\right\|&\le \left\|T_j (f_j-\widetilde{f_j})|\Lv\right\|\\&\le c \left\|f_j-\widetilde{f_j}|\Lv\right\|.
	\end{split}\end{align}
	Therefore, $\{T_j \}_j$ is (Lipschitz-)continuous and moreover we get the unique continuous extension $T:\Lciv\hra\Lv$ of $\{T_j \}_j$ using \eqref{FGM} in the same way as in the linear case.\\

\textit{Step 2}.
It remains to justify that  also $T: \Lciv\hra\Lciv$. 
 By means of a density argument we may assume that $\{f_j\}_j\in \Lciv$	with $f_j\in D(\rn)$ for all $j\in\nat_0$ and a $k$ such that $f_j=0$ for $|j|>k$. 
	There is an $\bar{R}$ such that $\supp f_j \subset B_{\bar{R}}(0)$ for all $j\in\nat_0$. Then 
	\begin{\eq} \label{43a_2:GG}
\left(\sum_{j=0}^{k}|(T_j f_j)(x)|^q\right)^\frac{1}{q} \le c \, |x|^{-n} \qquad \text{if} \quad |x| \ge 2 \bar{R}
\end{\eq}
  using (\ref{43a:GG}) and triangle inequality. Here the constant $c$ depends on $\{f_j\}_j$. 
	Let $R\ge 2 \bar{R}$. Then one has for cubes $Q_{JM}$ with $Q_{JM} \subset \{ x\in \rn: \, |x| >R \}$,
\begin{equation*}  
2^{J(\frac{n}{p} +r)} \left\|\left. \left(\sum_{j=0}^{k}|(T_j f_j)(x)|^q\right)^\frac{1}{q} \, \right| L_p (Q_{JM}) \right\| \le c \, 2^{Jr} \, R^{-n} \qquad \text{if} \quad J \in \nat_0.
\end{equation*}
Using in addition $1<p<\infty$ 
we obtain
\[ 
2^{J(\frac{n}{p} +r)} \left\|\left. \left(\sum_{j=0}^{k}|(T_j f_j)(x)|^q\right)^\frac{1}{q} \, \right| L_p (Q_{JM}) \right\| \le c \, 2^{J(\frac{n}{p} +r)} \, R^{-n(1- \frac{1}{p})} \text{ if } -J \in \nat.
\] 
Let $\psi_R\in D(\rn)$ be a smooth cut-off function with $\psi_R (x) =1$ if $|x| \le R$. Then $\psi_R \, T_j f_j \in C(\rn)$ compactly supported for $0\le j\le k$ (by triangle inequality using $T_j:D(\rn)\rightarrow L_\infty(\rn)$) and it follows 
\[
\lim_{R \to \infty} \| \{ T_j f_j - \psi_R T_j f_j \}_j | \Lv \| =0
\]
with $1<p<\infty$ and $0 < \frac{n}{p} +r <\frac{n}{p}$. Here one should mention 
that cubes which are not completely inside of $\{ x\in \rn: \, |x| >R \}$ are treated analogously and that  $T_j f_j=0$ for $j>k$ by $T_j 0=0$.
Hence $\{T_j f_j\}\in \Lciv$ by Proposition \ref{84A:GG} and therefore $T: \Lciv\hra\Lciv$ by the unique extension of $\{T_j\}_j$ to $\Lciv$.
Futhermore, we observe that the projection of $T$ to its $k-th$ component ($k\in\nat_0$) coincides with $\widetilde{T_k}$ where $\widetilde{T_k}:\Lci\hra\Lci$ is the unique continuous and bounded extension of $T_k$. This yields Assertion \ref{60:GG}.
\\\textit{Step 3}.
Finally, Assertions \ref{60b:GG} and \ref{59:GG} follow by duality (Theorems \ref{TDp1:GG} and \ref{TDp2:GG}).
As for the abstract background of duality one may consult \cite[pp. 112/113]{Yos80} and \cite[pp. 35/36]{Pie07}. 
We get firstly $\widetilde{T_j}':\Hr\hra\Hr$ and $\left(\widetilde{T_j}'\right)':\Lr\hra\Lr$ for all $j\in\nat_0$. Moreover, by the linearity of $T= \{\widetilde{T_j}\}_{j\in\nat_0}:\Lciv\hra \Lciv$ duality also implies  
$T': \Hv\hra\Hv$ as well as 
$(T')': \Lv\hra\Lv$.
The projection of $T'$ to its $k-th$ component ($k\in\nat_0$) coincides with $\widetilde{T_k}'$.
Indeed, let $f_j\equiv g_j\equiv 0$ for all $j\neq k$ and let $f_k, g_k\in D(\rn)$. By means of the  definition of $T'$ and $\widetilde{T_k}'$ we have
\begin{align*}
  &\int_{\rn} f_k(x)(T'(\{g_j\}))_k(x)\di x=
	\left\langle \{f_j\}, T'(\{g_j\})\right\rangle_{\left(\Lciv,\Hv\right)}
	\\=&\left\langle T(\{f_j\}), \{g_j\}\right\rangle_{\left(\Lciv,\Hv\right)}
	\\=& \left\langle \{\widetilde{T_j} f_j\},\{g_j\}\right\rangle_{\left(L_{p}(\ell_{q},w_{\alpha},\rn),L_{p'}(\ell_{q'},w_{-\alpha},\rn)\right)}
	=\int_{\rn} (\widetilde{T_k} f_k)(x)g_k(x)\di x
	\\=& \left\langle \widetilde{T_k} f_k,g_k\right\rangle_{\left(L_{p}(w_{\alpha},\rn),L_{p'}(w_{-\alpha},\rn)\right)}
	= \left\langle \widetilde{T_k} f_k,g_k\right\rangle_{\left(\Lci,\Hr\right)}
	\\=& \left\langle f_k, \widetilde{T_k}'g_k\right\rangle_{\left(\Lci,\Hr\right)}
	=\int_{\rn} f_k(x) (\widetilde{T_k}'g_k)(x)\di x
\end{align*} 
for $\alpha<-n/p$.
Analogously, using $\Hr\hra L_{-n/\vr}(\rn)$ we deduce that the projection of $(T')'$ to its $k-th$ component ($k\in\nat_0$) coincides with $\left(\widetilde{T_k}'\right)'$ on $D(\rn)$.
Moreover, we assume that the extensions of $T_j$ to $L_p(\rn)$ due to assumption \ref{vb2c:GG} are formally self-adjoint for all $j\in\nat_0$. Then we obtain   
\begin{align} \begin{split}\label{TG1}
&\left\langle f, \widetilde{T_j}'g\right\rangle_{\left(\Lci,\Hr\right)}=
\left\langle \widetilde{T_j} f,g\right\rangle_{\left(\Lci,\Hr\right)}
\\=& \left\langle {T_j} f,g\right\rangle_{\left(\Lci,\Hr\right)}
= \left\langle T_j f,g\right\rangle_{\left(\Lw,L_{p'}(w_{-\alpha},\rn)\right)} 
\\=&\int_{\rn} (T_j f)(x)g(x)\di x=
\left\langle T_j f,g\right\rangle_{\left(L_p,L_{p'}\right)}
  =\left\langle f,T_j g\right\rangle_{\left(L_p,L_{p'}\right)} 
\end{split}\end{align}
for all $f,g\in D(\rn)$ and $\alpha<-n/p$. Therefore, $\widetilde{T_j}'g= T_j g$ almost everywhere for all $g\in D(\rn)$ and $j\in\nat_0$ which means that $\widetilde{T_j}'$ are extensions of $T_j$ to $\Hr$.
Moreover, the biduals $\widetilde{T_j}'' = \left(\widetilde{T_j}'\right)'$, $j\in\nat_0$, 
are extensions of $T_j$ to $\Lr$ by 
\begin{align*}
&\left\langle g, \widetilde{T_j}''f\right\rangle_{\left(\Hr, \Lr\right)}
=\left\langle\widetilde{T_j}'g,f\right\rangle_{\left(\Hr, \Lr\right)}
=\left\langle\widetilde{T_j}'g,f\right\rangle_{\left(L_{u},L_{u'}\right)}
\\=&\int_{\rn} (\widetilde{T_j}' g)(x)f(x)\di x
=\left\langle f, \widetilde{T_j}'g\right\rangle_{(\Lci,\Hr)}
\\=&\left\langle \widetilde{T_j} f,g\right\rangle_{\left(\Lci,\Hr\right)}
=\left\langle {T_j} f,g\right\rangle_{\left(\Lci,\Hr\right)}
=  \left\langle T_j f,g\right\rangle_{\left(L_p,L_{p'}\right)}
\end{align*} 
for all $f,g\in D(\rn)$, $u=-{n}/{\vr}$ and $j\in\nat_0$. Therefore, $\widetilde{T_j}'' = T_j$ on $D(\rn)$ for $j\in\nat_0$. 
\end{proof}

\begin{Rem} \label{62:GG}
The extension in Part 3 of the theorem is not unique. There exist infinitely many extensions of $T_j$ acting in $L^r_p(\rn)$. This can be seen following  the same arguments as in \cite[Remark 5.3]{RoT14_2}.
Assumption \ref{vb2c:GG} can be replaced by 
\[
  T:L_{p}(\rn)\hra L_{p}(\rn)
\]
if $T_j=T$ for all $j$, $T$ is linear and if $q$ is between $2$ and $p$ (including $2$ and $p$) 
which holds by the fact that the $L_p$-boundedness of a linear operator implies the $L_{p}(\ell_q,\rn)$-boundedness for $q\in [p,2]$ for $p\le 2$ and $q\in [2,p]$ for $p>2$ cf. \cite[Corollary 4.5.4]{G04}.
\end{Rem}

\begin{Rem} \label{LS}
There are a lot of papers dealing with singular integrals in Morrey spaces. However, its well-definedness on the Morrey-type spaces under consideration as well as the norm estimates in these spaces have to be treated with greater care than usually done. On the one-hand one has to investigate how to extend singular integrals to Morrey spaces and on the other hand the estimates \eqref{43a:GG} are not available in general for functions belonging to Morrey spaces. Let us emphasize that we used \eqref{43a:GG} just for functions of $D(\rn)$.
The question if the estimate \eqref{43a:GG} holds for some singular integrals also for all $f\in \Lr$ leads to an investiation of its maximal truncated versions (cf. \cite[Prop. 2.25, Rem. 2.26]{T-HS} as well as Sections \ref{TGIS} and \ref{TGIS2}).
Indeed, for these reasons in many papers one can only find the weaker mapping property $T:\Lci\hra\Lr$ (see, for example \cite{FR93, N94}, and \cite{Y98, G11, Mu12} for operators satisfying \eqref{43a:GG}). 
In this sense our results on $\Lci$ and $\Lciv$ are new (including even the boundedness of the Hardy-Littlewood maximal operator), in particular with respect to their generality. Note that results for Calder\'{o}n-Zygmund operators in $\Lci$ (scalar case) have been proved already in \cite{RT13} and \cite{RoT14_2}. Moreover, some results in $\Hr$ and in $\Hv$ and $\Lv$ seem to be new.
The paper \cite{AX12} made the important observation that the bidual of the completion of $D(\rn)$ with respect to the Morrey norm coincides with the Morrey space itself (cf. \eqref{GiF1}) and provided the basis of our investigations. 
To overcome the above mentioned problems investigating Calder\'{o}n-Zygmund operators in $L_p^r(\rn)$ they considered Muckenhoupt weighted characterizations of Morrey spaces and their preduals. However, this approach has also some weak points with respect to norm estimates since it does not take into account that the operator norm of classical operators of harmonic analysis (as the Hilbert transform) in Muckenhoupt weighted spaces usually depends on the Muckenhoupt weight. 
\par
We want to refer also to a less known forerunner result which can be found in \cite{Alv96}. There a solution for the above mentioned difficulties has been given for some Calder\'{o}n-Zygmund operators in $H^{\vr}L_p (\rn)$.
\end{Rem}

\subsection{Calder\'{o}n-Zygmund operators}
\subsubsection{Duality approach}

\begin{Def} \label{LS5}
We define \textit{Calder\'{o}n-Zygmund operators} with homogeneous kernels with degree $-n$ setting, 
\begin{align*}
  (T^\Omega f)(y)&\equiv \text{ p.v. } \int_{\real^n} \frac{\Omega(z/|z|)}{|z|^n}f(y-z)\di z ,
	\end{align*}
where $f\in S(\rn)$ and $\Omega \in L_\infty({\cal S}^{n-1})$ with zero integral and ${\cal S}^{n-1}$ denotes the unit sphere.
\end{Def}

\begin{Corol} \label{vB:corol} 
Let $1<p<\infty$, $-\frac{n}{p}\leq r<0$, $-n < \vr < -\frac{n}{p'}$, $1< q<\infty$.
Then the following statements hold true. 
\begin{enumerate} \item
There are unique linear and bounded extensions of $T^\Omega$ to $\Lci$ and to $\Hr$ denoted again by $T^\Omega$ such that 
\begin{align*} 
&\left\{T^\Omega\right\}_{j\in\nat_0}: \Lciv\hra\Lciv \quad\text{and}\quad \\&
\left\{T^\Omega\right\}_{j\in\nat_0}: \Hv\hra\Hv.
\end{align*}
\item
There are infinitely many linear and bounded extensions of $T^\Omega$ to $\Lr$ denoted again by $T^\Omega$ such that
\[ \left\{T^\Omega\right\}_{j\in\nat_0}: \Lv\hra \Lv.\] 
\end{enumerate}
\end{Corol}

\begin{proof}
We observe that $T^\Omega: D(\rn)\rightarrow \text{Lip}(\rn)$. 
Indeed, by the same arguments as in \cite[proof of Step 2 of Thm. 1.1, p. 8]{RT13} we get the mapping properties $T^\Omega: W_p^k(\rn)\hra W_p^k(\rn)$ for Sobolev spaces which lead to the above assertion by means of Sobolev type embeddings. 
Moreover, for $\Omega \in L_\infty({\cal S}^{n-1})$ it holds $T^\Omega:L_p(\ell_q,\rn)\hra L_p(\ell_q,\rn)$ by \cite{DRF86}.
Now we obtain $T^\Omega: \Lv\hra\Lv$ for $\Omega \in L_\infty({\cal S}^{n-1})$ applying Theorem \ref{vB:thm}.
Note that the dual of the extension of $T^{\Omega(-\cdot)}$ to $\Lci$ coincides with $T^{\Omega}$ on $D(\rn)$ by the same arguments as in \eqref{TG1}.
\end{proof}

\subsubsection{Alternative approach using some Muckenhoupt weights} \label{TGIS}

The following alternative method due to Triebel \cite[Section 2.5.3, Prop. 2.25, Rem. 2.26]{T-HS} yields extensions of Calder\'{o}n-Zygmund operators which are bounded in $\Lr$. He studied the boundedness of $T^{\Omega}$ with $\Omega \in C^1({\cal S}^{n-1})$. Here we generalize his approach to some non-convolution type Calder\'{o}n-Zygmund operators. 
At first we observe that Morrey spaces $\Lr$ are continuously embedded into some Muckenhoupt weighted $L_p$-spaces.
Recall that $w_\alpha(\cdot)=(1+|\cdot|^2)^\frac{\alpha}{2}$, and that $L_p (\rn, w_\alpha)$ be the corresponding weighted $L_p$-space, normed as in \eqref{GiE}.
\begin{Prop}[Proposition 2.10 in \cite{T-HS}] \label{TttL}
Let $1<p<\infty$, $-\frac{n}{p}\leq r<0$, $-n<\alpha\,p<-n-rp$.
Then it holds
\begin{\eq} \label{74:GG}
L^r_p (\rn)  
\hra \Lw. 
\end{\eq}
\end{Prop}
\begin{proof}
Let $f\in\Lr$. Then \eqref{74:GG} follows from
\begin{align*}
  \int_{\rn} |f(x) w_\alpha(x)|^p \di x &\le c \left( \int_{|x|\le 1} |f(x)|^p \di x + \sum_{j\in\nat_0} 2^{j\alpha\,p} \int_{2^{j}\le |x|\le 2^{j+1}} |f(x)|^p \di x \right)
	\\ &\le \hat{c} \left( \int_{|x|\le 1} |f(x)|^p \di x + \sum_{j\in\nat_0} 2^{j(\alpha \,p+n+rp)} \left\|f|\Lr\right\|^p \right)
	\\ &\le \bar{c} \left\|f|\Lr\right\|^p.
\end{align*}
\end{proof}

\begin{Thm} \label{vB2:thm}
Let $1<p<\infty$, $-\frac{n}{p}< r<0$, $-n < \vr < -\frac{n}{p'}$. Let $T$ be an operator with domain $D(\rn)$ satisfying

		 \[\left\|Tf|L_{2}(\rn)\right\|
			\leq c_1 \left\|f|L_{2}(\rn)\right\|
	\] 
	where the constant $c_1$ is independent of $f\in D(\rn)$
and
	\begin{\eq} \label{2_59:GG}
	  (T f)(y) = \lim_{\varepsilon \searrow 0} \int_{z\in\real^n, |y-z|\ge \varepsilon} K(y,z) f(z)\di z  
	\end{\eq}
	almost everywhere for all $ f\in D(\rn)$, where the function $K(\cdot,\cdot)$ defined $\rn \times \rn \setminus \{(x,x):x\in\rn\} $ satisfies the conditions $|K(x,y)|\leq c_2 |x-y|^{-n}$ and 
\begin{align*}
  &|K(x,y)-K(x',y)|\leq c_2 \frac{|x-x'|^\delta}{(|x-y|+|x'-y|)^{n+\delta}},~\\&\qquad
  \text{ whenever }~ 2|x-x'|\le \max(|x-y|,|x'-y|), \\
  &|K(x,y)-K(x,y')|\leq c_2 \frac{|y-y'|^\delta}{(|x-y|+|x-y'|)^{n+\delta}},~\\&\qquad
	\text{ whenever }~ 2|y-y'|\le \max(|x-y|,|x-y'|).
\end{align*}
Then the following statements hold true.
\begin{enumerate} 
\item \label{2_60:GG}
There are linear and bounded extensions of $T$ to $\Lr$. 
\item \label{2_60b:GG}
There is an unique linear and bounded extension of $T$ to $\Lci$ and to $\Hr$.
\end{enumerate}
\end{Thm}

\begin{proof}
  By \cite[Cor. 9.4.7]{G09} there is an unique linear and bounded extension $\widetilde{T}$ of $T$ to $\Lw$ with $-n<\alpha\,p< n(p-1) $.  
	Therefore, Proposition \ref{TttL} yields $\widetilde{T}:\Lr\hra \Lw$. 
	We have even 
	\[
	  \sup_{\varepsilon > 0} \left| \int_{z\in\real^n, |y-z|\ge \varepsilon} K(y,z) f(z)\di z\right| : \Lw\hra\Lw
	\]
	by \cite[Thm. 9.4.6]{G09}. Together with \eqref{2_59:GG} we see that
	\[
	  (\widetilde{T} f)(y) = \lim_{\varepsilon \searrow 0} \int_{z\in\real^n, |y-z|\ge \varepsilon} K(y,z) f(z)\di z
	\]  
almost everywhere for all  $f\in\Lw$ by \cite[Thm. 2.1.14]{G09}. Now \eqref{43a:GG} holds for all $f\in\Lr$ with $y\notin \supp f$. As in Step 1 of the proof of Theorem \ref{vB:thm} we obtain 
$\widetilde{T}:\Lr\hra\Lr$ that is
Assertion \ref{2_60:GG}.
In particular, $\widetilde{T}:\Lci\hra\Lr$.
For $\alpha=0$ by \cite[Cor. 9.4.7]{G09} we have especially
\[
  \left\|Tf\left|L_{-\frac{n}{r}}(\rn)\right.\right\|\le c \left\|f\left|L_{-\frac{n}{r}}(\rn)\right.\right\|
\]
for all $f\in D(\rn)$ where the constant $c$ does not depend on $f$.
Hence, $T:D(\rn)\rightarrow L_{-{n}/{r}}(\rn)$.
Because of the embedding $L_{-{n}/{r}}(\rn) \hra \Lr$ and the density of $D(\rn)$ in $L_{-{n}/{r}}(\rn)$ we even have $T: D(\rn) \rightarrow \Lci$. 
Indeed, let $f\in D(\rn)$. Then $Tf\in L_{-{n}/{r}}(\rn)$ and thus there is a sequence of functions of $D(\rn)$ which tends to $Tf$ in $L_{-{n}/{r}}(\rn)$ and hence in $\Lr$ which shows $Tf \in \Lci$. 
Thus, $\widetilde{T}:\Lci\hra\Lci$.
The adjoint kernel of $K(x,y)$ given by $\ol{K(y,x)}$ also satisfies the required assumptions on the kernel. Hence, its corresponding operator is also bounded in $L_{p}(\rn)$ 
(cf. \cite[Def. 8.1.2]{G09}) but its dual coincides by the same arguments as in Step 3 of the proof of Theorem \ref{vB:thm} with the operator $T$ (with the kernel $K(x,y)$) on $D(\rn)$ which implies Assertion \ref{2_60b:GG}.
\end{proof}
\begin{Rem}
  Let us point out that for this method the embedding of $\Lr$ in some Muckenhoupt weighted space is crucial for extending the domain of the considered Calder\'{o}n-Zygmund operators to $\Lr$. 
  Recall the fact that the Hilbert transform is acting in $\Lw$ if, and only if, $w_\alpha$ is a Muckenhoupt weight. Moreover, we needed 
	\[
	  (\widetilde{T} f)(y) = \lim_{\varepsilon \searrow 0} \int_{z\in\real^n, |y-z|\ge \varepsilon} K(y,z) f(z)\di z  
	\]
almost everywhere for all $ f\in\Lr$.	This is a rather deep result in comparison to the $L_p$-boundedness which we require in Theorem \ref{vB:thm}.
Finally, let us emphasize again that the extension in Part 1	is by no means unique.
\end{Rem}
\subsection{Vector-valued maximal inequalities and maximal Calder\'{o}n-Zygmund operators} \label{TGIS2}
\begin{Def} 
We define \textit{maximal Calder\'{o}n-Zygmund operators} with homogeneous kernels with degree $-n$ by setting 
\begin{\eq*} 
  (T^\Omega_* f)(y)\equiv \sup_{\varepsilon > 0} \left|\int_{|z|\geq \varepsilon} \frac{\Omega(z/|z|)}{|z|^n}f(y-z)\di z\right|
\end{\eq*}
where $f\in \bigcup_{1\leq p<\infty} L_p(\rn)$ and $\Omega \in L_\infty({\cal S}^{n-1})$ with zero integral and ${\cal S}^{n-1}$ denotes the unit sphere.
As usual, the Hardy-Littlewood maximal operator $M$ 	is given by 
	\begin{equation*} 
		(M f)(y)\equiv\sup_{R>0} \frac{1}{|B_R(y)|} \int_{B_R(y)}\left|f(z)\right|\di z, \qquad f \in L_1^\text{loc}(\real^n). 
	\end{equation*}
\end{Def}

\begin{Rem}
  If $f\in \bigcup_{1\leq p<\infty} L_p(\rn)$ then 
	\[
	  \left|\int_{|z|\geq \varepsilon} \frac{\Omega(z/|z|)}{|z|^n}f(y-z)\di z\right|
	\]
	is bounded for each $\varepsilon >0$ and $y\in\rn$ by H\"older's inequality. Hence $(T^\Omega_* f)(y)$ is well-defined for all $y\in\rn$, but might be infinite.
\end{Rem}

\begin{Corol} \label{LS4}
Let $1<p<\infty$, $-\frac{n}{p}\leq r<0$, $-n < \vr < -\frac{n}{p'}$, $1< q<\infty$.
Then 
\begin{align}\begin{split} \label{TGIS3} &\left\{M\right\}_{j\in\nat_0}: \Lciv\hra\Lciv \quad\text{and}\quad \\&
\left\{M\right\}_{j\in\nat_0}: \Lv\hra\Lv.
\end{split}\end{align}
Moreover, if $\Omega \in C^1({\cal S}^{n-1})$, then
\begin{\eq} \label{TGIS5}
T^\Omega_*: \Lci\hra\Lci \quad\text{and}\quad 
T^\Omega_*: \Lr\hra\Lr.
\end{\eq} 
\end{Corol}

\begin{proof} 
At first we show that $M:D(\rn)\rightarrow \text{Lip}(\rn)$. Let $f\in D(\rn)$ and $f_h(\cdot)\equiv f(\cdot+h)$ for $h\in\rn$. By sublinearity of $M$ we obtain 
\[
  M f_h = M (f_h-f +f) \leq M (f_h-f) + Mf \text{ and thus } |M f_h - Mf|\leq M (f_h-f). 
\]
It follows that
\[
  |(M f)(x+h)-(Mf)(x)|= |(M f_h)(x) - (Mf)(x)|\leq [M (f_h-f)](x)\leq L h ,
\]
where $L$ is the Lipschitz constant of $f$ and $x\in\rn$. (We even showed $M:\text{Lip}(\rn)\rightarrow \text{Lip}(\rn)$ with the arguments due to \cite[Remarks 2.2]{K97}.)
A version of Cotlar's inequality leads to the estimate
\[
  (T^\Omega_* f)(x) \leq c ( [M(|T^\Omega f|)](x) + (Mf)(x) )
\]
for $x\in\rn$ (cf. \cite[Lemma 5.15]{D01}). As above we obtain
\[
  |(T^\Omega_* f)(x+h)-(T^\Omega_*f)(x)|= |(T^\Omega_* f_h)(x) - (T^\Omega_*f)(x)|\leq [T^\Omega_* (f_h-f)](x).
\]
Together with $T^\Omega: D(\rn)\rightarrow \text{Lip}(\rn)$ (cf. proof of Corollary \ref{vB:corol})
it follows from $M:\text{Lip}(\rn) \rightarrow \text{Lip}(\rn)$ also that $T^\Omega_*:D(\rn)\rightarrow \text{Lip}(\rn)$. 
Moreover, we claim that \eqref{43a:GG} holds also for $M$. Indeed, let $f\in D(\rn)$ with $y\notin \supp(f)$. Then there exists an $i\in\ganz$ such that $B_{2^i}(y)\cap \supp f = \emptyset$. Let ${f}^j\equiv\chi_{{B_{2^{j+1}}(y)}\setminus{B_{2^{j}}(y)}}{f}$ for $j\ge i$. Hence,
\begin{align*}
  |(Mf)(y)|&\le \sup_{R>0} \frac{1}{|B_R(y)|} \int_{B_R(y)} |f(z)| \di z \le \sum_{j=i}^\infty \sup_{R>0} \frac{1}{|B_R(y)|} \int_{B_R(y)} |f^j(z)| \di z 
	\\ &\le \sum_{j=i}^\infty \frac{1}{|B_{2^{j}}(y)|} \int_{\rn} |f^j(z)| \di z
	\le c \sum_{j=i}^\infty \int_{{B_{2^{j+1}}(y)}\setminus{B_{2^{j}}(y)}} \frac{|f(z)|}{2^{jn}} \di z
	\\ &\le c' \sum_{j=i}^\infty \int_{{B_{2^{j+1}}(y)}\setminus{B_{2^{j}}(y)}} \frac{|f(z)|}{|y-z|^n} \di z
	=c' \int_{\rn} \frac{|f(z)|}{|y-z|^n} \di z.
\end{align*}
Now Theorem \ref{vB:thm} implies the existence of unique continuous and bounded extensions of $M$ to $\Lciv$ and of $T^\Omega_*$ to $\Lci$. Since \eqref{43a:GG} also holds for $M$ and all $f\in\Lr$ in place of all
$f\in D(\rn)$ we achieve at \eqref{TGIS3} as in Step 1 of the proof of Theorem \ref{vB:thm}. Moreover,  
$T^\Omega_*$ is also well-defined on $\Lw$ if $-n<\alpha\,p<n(p-1)$ by \cite[Thm. 9.4.6]{G09}
(and not only on $\bigcup_{1\leq p<\infty} L_p(\rn)$) and hence on $\Lr$ by Proposition \ref{TttL}. Thus, \eqref{43a:GG} also holds for $T^\Omega_*$ and all $f\in\Lr$ in place of all $f\in D(\rn)$. This yields \eqref{TGIS5}. 
\end{proof}
\subsection{Fourier multipliers}
\subsubsection{Multipliers generated by characteristic and smooth functions} 

\begin{Corol} \label{LS3}
Let $1<p<\infty$, $-\frac{n}{p}\leq r<0$, $-n < \vr < -\frac{n}{p'}$, $1< q<\infty$.
Let $\{I_j\}_{j\in\nat_0}$ be a sequence of intervals on the real line, finite or infinite, and let $\{ S_j\}_j$ be the sequence of operators defined by 
\[
(S_j f)\,\hat{ }\,(\xi)=\chi_{I_j}(\xi) \hat{f}(\xi), ~~f\in D(\real),~ \xi\in\real. 
\]
Moreover, let $\psi\in S(\real^n)$ with $\psi(0)=0$. 
We define
\[
  \psi_j(\xi)=\psi(2^{-j}\xi) \quad \text{ and } \quad (\tilde{S_j} f)\,\hat{ }= \psi_j \hat{f} \quad \text{ for } \quad j\in\ganz, \ \xi\in\rn, \ f\in S'(\rn).
\]
Then the following statements hold true.
\begin{enumerate} 
\item \label{26:GiG1} 
There are unique linear and bounded extensions of $S_j$ to $\Lci$ and to $\Hr$ denoted again by $S_j$ and satisfying the mapping properties 
\begin{align*}\begin{split}
&\left\{S_j\right\}_{j\in\nat_0}: \Lcive\hra\Lcive \quad\text{and}\quad \\& \left\{S_j\right\}_{j\in\nat_0}: \Hve\hra\Hve
\end{split}\end{align*}
\item \label{26:GiG2}
There are infinitely many linear and bounded extensions of $S_j$ to $\Lr$ denoted again by $S_j$ such that
\[ \left\{S_j\right\}_{j\in\nat_0}: \Lve\hra \Lve.\] 
\item \label{26:GiG3}
We have the mapping properties
\begin{align*}
&\left\{\tilde{S_j}\right\}_{j\in\ganz}: \Lciv\hra\Lciv, \quad \\& \left\{\tilde{S_j}\right\}_{j\in\ganz}: \Hv\hra\Hv,
\end{align*}
and
\[
\left\{\tilde{S_j}\right\}_{j\in\ganz}: \Lv\hra\Lv~.
\] 
Here we used the notation $\ell_q=\ell_q(\ganz)$.
\end{enumerate}
\end{Corol}

\begin{proof}
Part \ref{26:GiG1} and part \ref{26:GiG2} are consequences of Theorem \ref{vB:thm} (see also Corollary \ref{vB:corol}).
The required $L_{p}(\ell_q,\rn)$-boundedness follows from \cite[Corollary 8.2]{D01}. Alternatively, it suffices the $L_{p}(\ell_q,\rn)$-boundedness  of the Hilbert transform $(H f)(y)\equiv\frac{1}{\pi} \lim_{\varepsilon \searrow 0} \int_{|z-y|\geq \varepsilon} \frac{f(z)}{y-z} \di z, f\in S(\rn)$ (see e.g. \cite[Cor. 4.6.3]{G04}) This can be seen using the formula
  \[
    S_j f_j = \frac{i}{2} \left(M_{a_j} H M_{-a_j} f_j - M_{b_j} H M_{-b_j} f_j\right),
  \]
 where $I_j=(a_j,b_j)$    (with the obvious modifications if the interval is unbounded) and where $M_{a} f (\cdot)\equiv e^{2\pi i a\cdot} f (\cdot)$ (\cite[(3.9)]{D01}). 
 Now the desired result follows from Theorem \ref{vB:thm} (for $n=1$) taking into account also that the dual of the extension of the multiplier generated by $-I_j$ coincides with $S_j$ on $D(\rn)$ by the same arguments as for \eqref{TG1}.  
\par
Moreover, we observe that
\begin{\eq} \label{26:GiG5}
  \left\{\tilde{S_j}\right\}_{j\in\ganz}: L_p(\ell_q,\rn)\hra L_p(\ell_q,\rn)~.
\end{\eq}
This follows, for example, from \cite[(8.1), page 158]{D01}. The needed H\"ormander condition is fulfilled by \eqref{GTL}. 
Indeed, we have 
\begin{\eq} \label{2:GiC}
  \left\|\left.\left\{ |\nabla \Psi_j (x)| \right\}_j\right|\ell_2\right\|\le \frac{c}{|x|^{n+1}}, \qquad x\in\rn
\end{\eq}
(cf. \cite[page 161]{D01}). H\"older's inequality yields 
\[
  |\Psi_j (x-y)-\Psi_j (x)|\le |y| \left(\int_{0}^1 |(\nabla \Psi_j) (x-ty)|^2   \di t \right)^\frac{1}{2}
\]
and furthermore using \eqref{2:GiC}
\begin{align}\begin{split} \label{GTL}
  \left\|\left\{\Psi_j (x-y)-\Psi_j (x)\right\}|\ell_2\right\|&\le |y| \left(\int_{0}^1 \left\|  \left|(\nabla \Psi_j) (x-ty) \right| |\ell_2\right\|^2  \di t   \right)^\frac{1}{2}
	\\&\le c\ \frac{|y|}{|x|^{n+1}}
\end{split}\end{align} 
for $|x|\geq 2 |y|$.
Using \eqref{26:GiG5}  we find
$\left\{\tilde{S_j}\right\}_{j\in\ganz}: \Lciv\hra\Lciv$ and $\left\{\tilde{S_j'}\right\}_{j\in\ganz}: \Hv\hra\Hv$
by means of Theorem \ref{vB:thm}. Here we have to show that
in particular assumption \eqref{43a:GG} is fulfilled. 
 If $\hat{\Psi}\equiv\psi$ and $\Psi_j(\cdot)\equiv 2^{jn}\Psi(2^j \cdot)$, then $\hat{\Psi}_j=\psi_j$ and $\tilde{S_j} f=\Psi_j*f \in C^\infty(\rn) \cap S'(\rn)$ by $\Lr\hra S'(\rn)$. 
In particular, $\tilde{S_j} f=\Psi_j*f$ makes sense pointwise for all $f\in\Lr$. 
Furthermore, 
\[ 
  \left(\sum_{j\in\ganz}\left|\Psi_j*f (x)\right|^2\right)^\frac{1}{2}\le \int_{\real} |f(y)| \left(\sum_{j\in\ganz}\left|\Psi_j (x-y)\right|^2\right)^\frac{1}{2} \di y
  \le c \int_{\real} \frac{|f(y)|}{|x-y|^n} \di y 
\] 
for all $x\in\rn$ and all $f\in \Lr$ with $x\notin\supp(f)$ where $\left\|\left\{\Psi_j (\cdot)\right\}|\ell_2\right\|\le c |\cdot|^{-n}$ (cf. \cite[page 161]{D01}).
This implies \eqref{43a:GG}. 
We note that dual of the extension of the multiplier of $\psi_j(-\cdot)$ coincides with $\tilde{S_j}$ on $D(\rn)$ by the same arguments as for \eqref{TG1}. Hence,
\[ 
  \left\{\tilde{S_j}\right\}_{j\in\ganz}: \Hv\hra\Hv.
\] 
We obtain 
\[ \left\{\tilde{S_j}\right\}_{j\in\ganz}: \Lv\hra\Lv \]
with the same norm estimates as in Theorem \ref{vB:thm}.
Hereby we emphasize that the operator $\tilde{S_j}$ is well-defined on $S'(\rn)$, in particular on $\Lr$. Moreover, \eqref{43a:GG} holds for $f\in \Lr$ in place of $f\in D(\rn)$.
\end{proof}

\begin{Rem} \label{GiR}
The vector-valued Fourier multiplier assertion proved Corollary \ref{LS3} paves the way to introduce predual Morrey versions $H^\vr A^s_{p,q}(\rn)$ of the Besov-Triebel-Lizorkin spaces $A^s_{p,q}(\rn)$. In particular it implies the independence of admitted resolutions of unity. 
One replaces the $L_p(\rn)$-norm in the definition of $A^s_{p,q}(\rn)$ by the $\Hr$-norm in order to define $H^\vr A^s_{p,q}(\rn)$. The vector-valued Fourier multiplier assertion in Corollary \ref{LS3} is also the key ingredient to obtain as in \cite[Section 2.3.3]{T83} the density of $S(\rn)$ in $H^\vr A^s_{p,q}(\rn)$. Moreover as in \cite[Section 2.11.2]{T83} one can show using our vector-valued duality assertions (Theorem \ref{TDp1:GG} and \ref{TDp2:GG}) also that the dual of $H^\vr A^{-s}_{p',q'}(\rn)$ is $L^r A^s_{p,q}(\rn)$ and furthermore that the dual of the completion of $S(\rn)$ with respect to $L^r A^s_{p,q}(\rn)$ is $H^\vr A^{-s}_{p',q'}(\rn)$. Here $L^r A^s_{p,q}(\rn)$ stands for the morreyfied versions of $A^s_{p,q}(\rn)$ which are defined by replacing the $L_p(\rn)$-norm in the definition of $A^s_{p,q}(\rn)$ by the $L^r_p(\rn)$-norm.  
\par
Moreover, in the one-dimensional case ($n=1$) Corollary \ref{LS3} implies  also Lizorkin representations of the Triebel-Lizorkin-Morrey spaces $L^r A^s_{p,q}(\real)$ (cf. \cite[Section 2.5.4]{T83}) using in addition Nikol'skij inequalities for Morrey spaces published in \cite[Thm. 2.2.9, Thm. 2.2.20]{Ros13}. We want to mention that the spaces $L^r A^s_{p,q}(\rn)$ are  studied, in particular, in \cite{YSY10, HS11, R, Tr12, Ros13}).
\end{Rem}

\subsubsection{Strongly singular integrals}
\begin{Def} 
Let $0<b<1$ and let $\varphi$ be a smooth  
cut-off function with $\varphi=1$ on $\{|\xi|\geq 1 \}$ and $\varphi=0$ on $\{|\xi|\leq 1/2 \}$.
If $f\in S(\rn)$, then we define \textit{strongly singular integrals}  as
\[
  (T_b f)(x) \equiv \int_{\xi\in\real^n} \frac{e^{i|\xi|^b}}{|\xi|^{nb/2}}\varphi(|\xi|) \hat{f}(\xi)e^{2\pi i x \xi} \di \xi
\]
(cf. \cite[p. 192]{SW94}).
\end{Def}

\begin{Corol} \label{LS2} 
Let $1<p<\infty$, $-\frac{n}{p}\leq r<0$, $-n < \vr < -\frac{n}{p'}$ and let $q\in [p,2]$ for $p\le 2$ and $q\in [2,p]$ for $p>2$. 
Then the following statements hold true.
\begin{enumerate} 
\item There are unique linear and bounded extensions of $T_b$ to $\Lci$ and to $\Hr$ denoted again by $T_b$ and satisfying 
\begin{align*}
&\left\{T_b\right\}_{j\in\nat_0}: \Lciv\hra\Lciv \quad\text{and}\quad\\& \left\{T_b\right\}_{j\in\nat_0}: \Hv\hra\Hv
\end{align*}
\item 
There are infinitely many linear and bounded extensions of $T_b$ to $\Lr$ denoted again  by $T_b$ such that
\[ \left\{T_b\right\}_{j\in\nat_0}: \Lv\hra \Lv.\] 
\end{enumerate}
\end{Corol}
\begin{proof} 
  
$T_b$ satisfies \eqref{43a:GG} by \cite[p. 192]{SW94}, see also \cite[Chapt. 5, Sect. 6.8]{D01}. 
The strongly singular integrals $T_b$ are bounded on $L_{p}(\real^n)$, $1<p<\infty$, by \cite[Section 6.8]{D01} and the references given there. 
Moreover, $T_b:W^k_p (\rn)\hra W^k_p (\rn)$ for all $k \in \nat$ using the lift operator (which is the Fourier multiplier corresponding to $(1+|\cdot|^2)^{\sigma/2}$ for $\sigma\in\real$) and, in particular,
it follows that $T_b f\in C^\infty$ for $f\in D(\rn)$ by well-known Sobolev embeddings. 
Having in mind Remark \ref{62:GG} we obtain the assertion by Theorem \ref{vB:thm}. 
\end{proof} 

\subsubsection{Bochner-Riesz multipliers}

\begin{Def}
Let $\lambda >0 $ and let $f\in S(\rn)$. We define \textit{Bochner-Riesz multipliers} 
as 
\[
  (B^\lambda f)(x) \equiv \int_{|\xi|\leq 1} (1-|\xi|^2)^\lambda \hat{f}(\xi)e^{2\pi i x \xi} \di \xi~.
\]
\end{Def} 

It is well-known that $B^\lambda f$ can be reformulated as
\begin{\eq} \label{55:GG}
  (B^\lambda f)(x) = c \lim_{\varepsilon \searrow 0} \int_{|x-y|\geq \varepsilon} \frac{J_{n/2+\lambda}(2\pi|x-y|)}{|x-y|^{n/2+\lambda}}f(y) \di y
\end{\eq}
where $J_\alpha$ stands for the Bessel function (cf. \cite[Lemma 8.18]{D01} or \cite[(10.2.1)]{G04}). 

\begin{Corol} 
If $\lambda \ge (n-1)/2$ then the statements of Corollary \ref{LS2} hold with  $B^\lambda$ in place of $T_b$.
\end{Corol}

\begin{proof} Let $\lambda = (n-1)/2$ be the critical index.
  Using   \eqref{55:GG} as well as the estimate $J_{n/2+\lambda}(|x|)\leq {c}{|x|^{-1/2}}$ (see e.g. \cite[Appendix B.6]{G04}) we see that $B^\lambda$ satisfies \eqref{43a:GG}. 
The $L_p$-boundedness of $B^\lambda$ at the critical index for $1<p<\infty$ is known (cf. \cite[Thm 8.15]{D01}). 
We also have $B^\lambda: D(\rn)\rightarrow \text{Lip}(\rn)$ by the same arguments as in \cite[proof of Step 2 of Thm. 1.1, p. 8]{RT13} taking into account the convolution structure of $B^\lambda$. Thus 
the assertion for $\lambda = (n-1)/2$ is a consequence of Theorem \ref{vB:thm}. Let $\lambda > (n-1)/2$. Then $|(B^\lambda f)|$ can be dominated pointwise by the Hardy-Littlewood maximal function $Mf$ 
for $f\in D(\rn)$ (cf. \cite[Exercise 10.2.8]{G04}). Hence,
  	\[
	  \left\|B^\lambda f_j|\Lv\right\|\le c\left\|f_j|\Lv\right\|
	\]
	where $c$ does not depend on $\{f_j\}_{j=0}^\infty\in \Lciv$ with $f_j\in D(\rn)$ for all $j$.
	As above we have $B^\lambda: D(\rn)\rightarrow \text{Lip}(\rn)$. As in the proof of Theorem \ref{vB:thm} we find an unique extension of $B^\lambda$ denoted again as $B^\lambda$ such that
	\[
	  B^\lambda:\Lciv\hra \Lciv.
	\]
	Hereby, we mention that the constant in \eqref{43a_2:GG} is allowed to depend on the fixed sequence of functions. The proof of the remaing parts of the corollary for $\lambda > (n-1)/2$ follows the same lines 
	as in the proof of Theorem \ref{vB:thm}. 
\end{proof}



\begin{thebibliography}{xxxxxxxxx}


\bibitem[AX04]{AX04}	\textsc{Adams, D. R. and Xiao , J.}: Nonlinear potential analysis on Morrey spaces and their
capacities. Indiana Univ. Math. J. \textbf{53}: No.6, 1629-1663 (2004).

\bibitem[AX12]{AX12}	\textsc{Adams, D. R. and Xiao , J.}: Morrey spaces in harmonic analysis. Ark. Mat. \textbf{50}, 201-230 (2012).

\bibitem[Alv96]{Alv96} J. Alvarez. Continuity of Calder\'{o}n-Zygmund type operators on the predual of a Morrey space. In: Clifford algebras in analysis and
related topics, CRC Press, Boca Raton, 1996, 309-319.

\bibitem[FR93]{FR93} \textsc{Di Fazio, G. and Ragusa, M.}: Interior estimates in Morrey spaces for strong solutions to nondivergence form equations with discontinuous coefficients. J. Funct. Anal. \textbf{112}, 241-256 (1993).

\bibitem[DYZ98]{Y98} \textsc{Ding, Y.; Yang, D.; Zhou, Z.}:
Boundedness of sublinear operators and commutators on Morrey spaces.
Yokohama Math. J. \textbf{46}, No.1, 15-27 (1998).

\bibitem[DRF86]{DRF86} \textsc{Duoandikoetxea, J. and Rubio de Francia, J. L.}:
Maximal and singular integral operators via Fourier transform estimates.
Invent. Math. \textbf{84}, 541-561 (1986).

\bibitem[Duo01]{D01} \textsc{Duoandikoetxea, J.}:
Fourier analysis. Transl. from the Spanish and revised by David Cruz-Uribe. (English)
Graduate Studies in Mathematics. \textbf{29}. Providence, RI: American Mathematical Society (2001).

\bibitem[ET96]{ET96} \textsc{Edmunds, D. E. and Triebel, H.}:
Function spaces, entropy numbers, differential operators. Paperback reprint
of the hardback edition 1996. 
Cambridge Tracts in Mathematics \textbf{120}. Cambridge: Cambridge University Press.

\bibitem[GM13]{GM13} \textsc{Gogatishvili, A. and Mustafayev, R.Ch.}: 
New pre-dual space of Morrey space. 
J. Math. Anal. Appl. {\bfseries 397}, No. 2 (2013), 678-692.

\bibitem[Gra04]{G04}	\textsc{Grafakos, L.}:
Classical and modern Fourier analysis. Upper Saddle River, NJ : Pearson/Prentice Hall. (2004).

\bibitem[Gra09]{G09}	\textsc{Grafakos, L.}:
Modern Fourier analysis. 2nd ed. Grad. Texts in Math. \textbf{250}. Springer, New York 2009.

\bibitem[GAKS11]{G11}	\textsc{Guliyev, V. S.; Aliyev, S. S.; Karaman, T.; Shukurov, P. S.}:
Boundedness of Sublinear Operators and Commutators on Generalized Morrey Spaces.
Integr. Equ. Oper. Theory \textbf{71}, 327-355 (2011).

\bibitem[HS12]{HS11}	\textsc{Haroske, D. D. and Skrzypczak, L.}: Continuous embeddings of Besov-Morrey function spaces. Acta Math. Sin., Engl. Ser. \textbf{28}, No. 7, 1307-1328 (2012).

\bibitem[Kal98]{Kal98} \textsc{Kalita, E.A.}: Dual Morrey spaces. Dokl. Akad. Nauk {\bfseries 361} (1998), 447-449 (Russian); Engl. transl.: Dokl. Math. 
{\bfseries 58} (1998), 85-87.


\bibitem[Kin97]{K97} \textsc{Kinnunen, J.}:
The Hardy-Littlewood maximal function of a Sobolev function. 
Isr. J. Math. \textbf{100}, 117-124 (1997).


\bibitem[Mus12]{Mu12}	\textsc{Mustafayev, Rza Ch.}: 
On boundedness of sublinear operators in weighted Morrey spaces.
Azerb. J. of Math. \textbf{2}, 66-79 (2012).

\bibitem[Nak94]{N94}	\textsc{Nakai, E.}: Hardy-Littlewood maximal operator, singular integral operators and the
Riesz potentials on generalized Morrey spaces. 
Math. Nachr. \textbf{166}, 95-103 (1994).

\bibitem[NS14]{NS14}	\textsc{Nakai, E. and Sobukawa, T.}: $B_w^u$-function spaces and their interpolation. arXiv:1410.6327, p. 43, (2014).

\bibitem[Pie07]{Pie07} \textsc{Pietsch, A.}: History of Banach spaces and linear operators. Birkh\"{a}user, Boston, 2007.

\bibitem[Ros12]{R} \textsc{Rosenthal, M.}: Local means, wavelet bases and wavelet isomorphisms in Besov-Morrey and Triebel-Lizorkin-Morrey spaces. Mathematische Nachrichten. \textbf{286}, No. 1, 59-87 (2013).

\bibitem[Ros13]{Ros13} \textsc{Rosenthal, M.}: Mapping properties of operators in Morrey spaces and wavelet isomorphisms in related Morrey smoothness spaces. PhD-Thesis, Jena, 2013.

\bibitem[RT13]{RT13} \textsc{Rosenthal, M. and Triebel, H.}: Calder\'{o}n-Zygmund operators in Morrey spaces. Rev. Mat. Comp. \textbf{27}, 1-11 (2013).

\bibitem[RT14]{RoT14_2} \textsc{Rosenthal, M. and Triebel, H.}: Morrey spaces, their duals and preduals. Rev. Mat. Complut. DOI 10.1007/s13163-013-0145-z, p. 30 (2014).

\bibitem[SW94]{SW94} \textsc{Soria, F. and Weiss, G.}: A remark on singular integrals and power weights. Indiana Univ. Math. J. \textbf{43}, No.1, 187-204 (1994).

\bibitem[Tri78]{T78}	\textsc{Triebel, H.}:            
Interpolation theory. Function spaces. Differential operators. 
Deutscher Verlag des Wissenschaften, Berlin 1978.

\bibitem[Tri83]{T83}	\textsc{Triebel, H.}: Theory of Function Spaces. Birkh\"auser, Basel 1983.

\bibitem[Tri13]{Tr12} \textsc{Triebel, H.}: Local function spaces, heat and Navier-Stokes equations. Z\"urich: European Mathematical Society, 2013.    

\bibitem[Tri14]{T-HS}
\textsc{Triebel, H.}:
Hybrid function spaces, heat and Navier-Stokes equations. Z\"urich: European Mathematical Society, submitted.

\bibitem[Yos80]{Yos80} \textsc{Yosida, K.}: Functional analysis, 6th edition. Springer, Berlin, 1980.

\bibitem[YSY10]{YSY10}	\textsc{Yuan, W.; Sickel, W.; Yang, D.}: Morrey and Campanato meet Besov, Lizorkin and Triebel. Springer, Heidelberg 2010.

\bibitem[Zor86]{Zor86} \textsc{Zorko, C.T.}: Morrey spaces. Proc. Amer. Math. Soc. {\bfseries 98}, 586-592 (1986).

\end{thebibliography}

\end{document}